\documentclass[12pt]{article}
\usepackage{amsthm,amsmath,indentfirst,amssymb}
\usepackage{cite}
\usepackage{makecell}
\usepackage{array}
\usepackage{multirow,booktabs}
\usepackage{amsfonts}
\usepackage{float}
\usepackage{graphicx}
\usepackage{amsmath}
\usepackage{mathrsfs}
\usepackage{epstopdf}
\usepackage{subcaption}
\usepackage{color}

\makeatletter
\def\thanks#1{\protected@xdef\@thanks{\@thanks
        \protect\footnotetext{#1}}}
\makeatother
\numberwithin{equation}{section}
\newtheorem{thm}{Theorem}[section]
\newtheorem{lem}{Lemma}[section]

\newtheorem{rem}{Remark}[section]
\allowdisplaybreaks[4]
\begin{document}
\title{\bf A priori error estimates for stable generalized finite element discretization of parabolic interface optimal control problems}
\author{Xindan Zhang$^{1}$, Jianping Zhao$^{1,2,*}$, Yanren Hou$^{1,3}$\\
\small \emph{$^{1}$ College of Mathematics and System Science, Xinjiang University, Urumqi 830046, China}\\
\small \emph{$^{2}$ Institute of Mathematics and Physics, Xinjiang University, Urumqi 830046, China}\\
\small \emph{$^{3}$ School of Mathematics and Statistics, Xi'an Jiaotong University, Xi'an 710049, China}\\}
\thanks{$^{*}$ Corresponding author. E-mail: zhaojianping@126.com}
\date{}
\maketitle \baselineskip 18pt
\begin{minipage}{150mm}{\small{\bf Abstract}
In this paper, we investigate optimal control problems governed by the parabolic interface equation, in which the control acts on the interface. The solution to this problem exhibits low global regularity due to the jump of the coefficient across the interface and the control acting on the interface. Consequently, the traditional finite element method fails to achieve optimal convergence rates when using a uniform mesh. To discretize the problem, we use fully discrete approximations based on the stable generalized finite element method for spatial discretization and the backward Euler scheme for temporal discretization, as well as variational discretization for the control variable. We prove a priori error estimates for the control, state, and adjoint state. Numerical examples are provided to support the theoretical findings.

}
\end{minipage}
\maketitle
\section{\bf Introduction}
In this paper, we consider the following optimal control problem:
\begin{equation}\label{s11}
\min\limits_{u\in U_{ad}} J(y,u)=\frac{1}{2}\int_{0}^{T}\int_{\Omega} (y-y_{d})^{2}dxdt+\frac{\alpha}{2}\int_{0}^{T}\int_{\Gamma} u^{2}ds dt
\end{equation}
subject to
\begin{equation}\label{s12}
\left\{\begin{array}{ll}
 ~y_{t}-\nabla\cdot(\beta\nabla y)  = f, & ~\text{in}~ \Omega\setminus\Gamma\times(0,T),  \\
 ~[y]_\Gamma=0,~[\beta\partial_{\mathbf{n}}y]_{\Gamma}=g+u, &~\text{on}~\Gamma\times (0,T),\\
 ~y=0,&~\text{on}~ \partial\Omega\times(0,T),\\
 ~y(0)=y_{0},& ~\text{in}~ \Omega,
\end{array}\right.
\end{equation}
where  $\Omega \subset \mathbb{R}^2$ is a bounded, convex open domain with Lipschitz continuous boundary separated by an $\mathcal{C}^2$ interface $\Gamma$. We assume that the interface $\Gamma$ divides the domain $\Omega$ into two subdomains $\Omega^{+}$ and $\Omega^{-}$, and $\Omega^{-}$ lies strictly inside $\Omega$ (see Fig. 1). The symbol $[v]_{\Gamma}=v^{-}|_{\Gamma}-v^{+}|_{\Gamma}$ denotes the jump of the function $v$ across the interface $\Gamma$ and the operator $\partial_{\mathbf{n}}$ denotes the normal derivative on $\Gamma $, i.e., $\partial_{\mathbf{n}}y=\mathbf{n} \cdot\nabla y$, where $\mathbf{n}$ is the unit normal direction of  $\Gamma$ pointing to $\Omega^{+}$. Let $\beta$ be a piecewise positive constant function given by
\begin{equation}\label{s13}
\beta =
\left\{\begin{array}{ll}
\beta^{+}, \qquad \text{in } \Omega^{+} ,\\
\beta^{-}, \qquad \text{in } \Omega^{-} .
\end{array}\right.
\end{equation}

The admissible controls set is given by
$$ U_{ad}=\{u\in L^{2}(0,T;L^{2}(\Gamma)):~u_{a}\leq u(x,t)\leq u_{b}, ~a.e. ~\text{on}~\Gamma\times(0,T)\},$$
where $u_{a} \leq u_{b}$. 
\begin{figure}[H]
\centering
	\begin{subfigure}{0.5\linewidth}
		\centering
		\includegraphics[width=2.5in,height=1.8in]{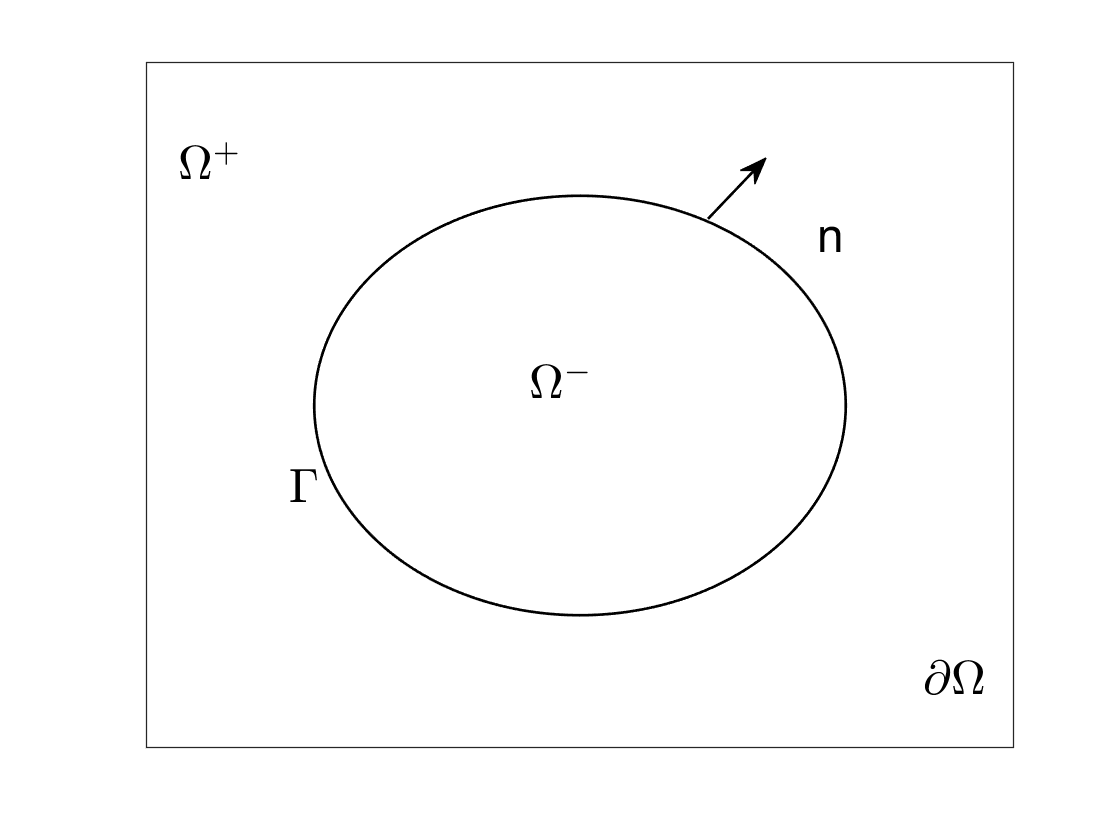}
	\end{subfigure}
 \caption*{Fig. 1. A geometry shape for the interface problem.}
   \label{interface}
\end{figure}
The optimal control of partial differential equations (PDEs) with interfaces plays a crucial role in various applications, including composite materials\cite{zhang2020}, crystal growth\cite{Meyer2006,Meyer2009}, and tumor growth\cite{Sprekels2021}. There are already many studies on numerical methods for elliptic interface optimal control problems, including the immersed finite element method\cite{Zhang2015,Gong2020,Wang2022}, the Nitsche-eXtended finite element method \cite{Wang2019}. In addition, the immersed finite element method was applied to parabolic optimal control problems with interfaces in \cite{Zhang20201}.

However, all of the work mentioned above has focused on distributed control, with few published results on the topic of interface control. This is because interface control is more complex than the case of distributed control both in theoretical analysis and numerical approximation. The authors of \cite{Wachsmuth2016} investigated $hp$-finite elements for elliptic interface optimal control problems with the control acting on the interface. The error estimates of order $O(h^{\frac{3}{2}})$ and $O(h)$ under different regularity assumptions for such problems were derived in \cite{Yang2018}. The authors in \cite{Su2023} employed the immersed finite element method to solve this type of problem and conducted a large number of numerical experiments to verify the effectiveness of this numerical method. In \cite{Lai2025}, the authors proposed the hard-constraint PINNs method for solving optimal control problems subject to PDEs with interfaces and control constraints, and performed extensive tests on various elliptic and parabolic interface optimal control problems to verify the effectiveness of the proposed methods. However, to the best of our knowledge, these references mainly focus on elliptic problems and seem to have made no contribution to the parabolic problems.

The main goal of this paper is to analyze the stable generalized finite element approximation of interface parabolic optimal control problems with the control acting on the interface. The main difficulty is the low regularity of the state variable on the whole domain caused by the fact that the jump of the coefficient across the interface and the control act on the interface. We use the stable generalized finite element for the space discretization of the state, while the backward Euler scheme is used for time discretization. For the discretization of the control variable, we use the variational discretization approach (see \cite{Hinze2005}). We derive a priori error estimates for the control, state, and adjoint state and then use numerical experiments to support our theoretical results.

The rest of the paper is organized as follows. In Section 2 we discuss the optimality conditions for the control problem and the corresponding regularity results. In Section 3 we present the discretization of the optimal control problem based on the variational discretization approach and the stable generalized finite element method. In Section 4, we derive an a priori error estimates for the control, state, and adjoint state. Finally, in Section 5 we provide some numerical examples to support our theoretical results.
\section{\bf Optimality system}
For $m\geq0$ and $1\leq q\leq \infty$, we denote the usual Sobolev space by $W^{m,q}(\Omega)$ with norm $\parallel \cdot\parallel_{m,q,\Omega} $ and semi-norm $\mid\cdot\mid_{m,q,\Omega}$. In particular, for $q=2$ we denote $H^{m}(\Omega)=W^{m,2}(\Omega)$ and $\|\cdot\|_{m,\Omega}=\|\cdot\|_{m,2,\Omega}$. Note that $L^{2}(\Omega)=H^{0}(\Omega)$ and $H_{0}^{m}(\Omega)=\{v\in H^{1}(\Omega): v=0 ~\text{on}~ \partial \Omega\}$.
\par
We denote by $L^{r}(0,T;W^{m,q}(\Omega))$ the Banach space of all $L^{r}$ integrable functions from $(0,T)$ to $W^{m,q}(\Omega)$ with the norm
$$\| v\|_{L^{r}(0,T;W^{m,q}(\Omega))}=\Big(\int_{0}^{T} \| v\|_{m,q,\Omega}^{r}dt\Big)^{\frac{1}{r}}\quad \mbox{for} \quad 1\leq r<\infty,$$
and standard modification for $r=\infty$. We denote inner products of the $L^{2}(\Omega)$ and $L^{2}(\Gamma)$ by
$$(v,w)=\int_{\Omega}vwdx\quad\forall~v,w\in L^{2}(\Omega)$$
and
$$\langle v,w\rangle_{\Gamma}=\int_{\Gamma}vwds\quad\forall~v,w\in L^{2}(\Gamma),$$
respectively.
For the subsequent analysis, we also need to define the following spaces:
$$W^{m,q}(\Omega^{+}\cup \Omega^{-})=\{v\in L^{2}(\Omega)|~v|_{\Omega^{+}}\in W^{m,q}(\Omega^{+}), v|_{\Omega^{-}}\in W^{m,q}(\Omega^{-}) \},$$
equipped with the norm
$$\|v\|_{m,q,\Omega^{+}\cup \Omega^{-}}=\|v\|_{m,q,\Omega^{-}}+\|v\|_{m,q,\Omega^{+}}. $$

The following regularity result for the interface problem (\ref{s12}) can be found in, e.g. \cite{Chen1998,Sinha2007,Sinha2005,Huang2002}.
\begin{lem} 
Assume that $f\in H^{1}(0,T;L^{2}(\Omega))$, $y_{0}\in H_{0}^{1}(\Omega)$ and $g+u\in L^{2}(0,T;H^{\frac{1}{2}}(\Gamma))$. Then there exists a unique solution $$y \in L^{2}(0,T;H^{2}(\Omega^{+}\cup\Omega^{-}))\cap H^{1}(0,T;H^{1}(\Omega^{+}\cup\Omega^{-})).$$
\end{lem}
Moreover, there holds the following maximal regularity result for the solution of this equation (see, e.g., \cite{Amann2021,Pruss2016}).
\begin{lem} Let $1<q<\infty$ with $q \notin \{3/2, 3\}$. Assume that $\Omega$ is a bounded domain with smooth boundary, let $f\in L^{q}(0,T;L^{q}(\Omega))$, $y_{0}\in B^{2-\frac{2}{q}}_{q,q}(\Omega^{+}\cup\Omega^{-})$ and $g+u\in L^{q}(0,T;W^{1-\frac{1}{q},q}(\Gamma)) \cap W^{\frac{1}{2}-\frac{1}{2q},q}(0,T;L^{q}(\Gamma))$. Suppose that the following compatibility conditions are satisfied:
\begin{equation*}
\left\{\begin{array}{ll}
~[\beta\partial_{\mathbf{n}}y_{0}]_{\Gamma}=g(0)+u(0), &~\text{if} ~q> 3,\\
~[y_{0}]_\Gamma=0,~y_{0}|_{\partial \Omega}=0, &~\text{if} ~q> 3/2.
\end{array}\right.
\end{equation*}
Then there exists a unique solution $$y \in L^{q}(0,T;W^{2,q}(\Omega^{+}\cup\Omega^{-}))\cap W^{1,q}(0,T;L^{q}(\Omega)).$$
\end{lem}
\noindent Here and after, $B^{2-\frac{2}{p}}_{q,p}(\Omega)$ stands for the Besov space defined by
$$B^{2-\frac{2}{p}}_{q,p}(\Omega)=[L^{q}(\Omega),W^{2,q}(\Omega)]_{1-1/p,p}.$$

To introduce the weak formulation of the equation (\ref{s12}), we define the bilinear form $ a(\cdot,\cdot): H^{1}(\Omega)\times H^{1}(\Omega)\rightarrow \mathbb{R} $ by
$$a(v,w)=\int_{\Omega}\beta\nabla v\cdot \nabla w dx \quad \forall v,w\in H^{1}(\Omega).$$
Let $\varepsilon (\Omega)$ be the energy space given by
$$\varepsilon(\Omega):=\{v\in H^{1}(\Omega):\|v\|^{2}_{\varepsilon,\Omega}:=a(v,v)<\infty\}.$$
The standard weak formulation of the state equation (\ref{s12}) is then defined as follows: Find a state $y(u) \in H^{1}_{0}(\Omega)$ satisfing
\begin{equation}\label{s21}
(y_{t}(u),w)+a(y(u),w)=(f,w)+\langle u+g, w\rangle_{\Gamma} \quad \forall  w\in H^{1}_{0}(\Omega), ~t\in (0,T),\\
\end{equation}
with $y(u)(0)=y_{0}$.
\par
For any given $u \in L^{2}(I;L^{2}(\Gamma))$, we can obtain that the state equation (\ref{s21}) admits a unique solution $ y(u)$. Therefore, we denote the control-to-state mapping of the state equation by $y:= Su$. The optimal control problem (\ref{s11}) can then be equivalently reformulated as
\begin{equation}\label{s22}
\min\limits_{u\in U_{ad}}J(u)= \frac{1}{2}\int_{0}^{T}\int_{\Omega} (Su-y_{d})^{2}dxdt+\frac{\alpha}{2}\int_{0}^{T}\int_{\Gamma} u^{2}ds dt.
\end{equation}
By standard arguments (see, e.g., \cite{Lions1971}), we can prove that the problem (\ref{s22}) admits a unique solution $\overline{u}\in U_{ad}$ with the corresponding state $\overline{y}=S\overline{u}$. Moreover, we have the following first-order optimality condition.
\begin{lem} Assume that $\overline{u}\in L^{2}(I;L^{2}(\Gamma))$ is the unique solution of problem \textup{(\ref{s22})} and let $\overline{y}$ be the associated state, there exists a unique adjoint state $\overline{p} \in L^{2}(I; H^{1}_{0}(\Omega))\cap H^{1}(I;L^{2}(\Omega))$ satisfying the adjoint equation
\begin{equation}\label{s23}
\left\{\begin{array}{ll}
 ~-\overline{p}_{t}-\nabla\cdot(\beta\nabla \overline{p})  = \overline{y}-y_{d}, & ~\text{in}~ \Omega\setminus\Gamma\times(0,T),  \\
 ~[\overline{p}]_\Gamma=0,~[\beta\partial_{\mathbf{n}}\overline{p}]_{\Gamma}=0,& ~\text{on}~ \Gamma\times(0,T), \\
 ~\overline{p}=0,&~\text{on}~ \partial\Omega\times(0,T),\\
 ~\overline{p}(T)=0,& ~\text{in}~\Omega,
\end{array}\right.
\end{equation}
and the variational inequality
\begin{equation}\label{s24}
\int_{0}^{T}\int_{\Gamma}(\alpha \overline{u}+\overline{p})(v-\overline{u})dsdt\geq 0,\quad \forall v\in U_{ad}.
\end{equation}
Moreover, the variational inequality is equivalent to
\begin{equation}\label{s25}
\overline{u}=P_{U_{ad}}\Big(-\frac{1}{\alpha}\overline{p}\Big|_{\Gamma}\Big),
\end{equation}
where $P_{U_{ad}}$ denotes the projection onto $U_{ad}$.
\end{lem}
\begin{proof}
 Since the optimal control problem is quadratic and convex, by the standard method as in \cite{Lions1971,Hinze2009}, the optimality condition reads
 $$J'(\overline{u})(v-\overline{u})=\int_{0}^{T}\int_{\Omega}\widetilde{y}(\overline{y}-y_{d})dxdt+\int_{0}^{T}\int_{\Gamma}\alpha \overline{u}(v-\overline{u})dsdt\geq 0\quad \forall v\in U_{ad},$$
 where $\widetilde{y}=S'(\overline{u})(v-\overline{u})$ is the solution of the equation
\begin{equation}\label{s26}
\left\{\begin{array}{ll}
 ~\widetilde{y}_{t}-\nabla\cdot(\beta\nabla \widetilde{y})  = 0, & ~\text{in}~ \Omega\setminus\Gamma\times(0,T),  \\
 ~[\widetilde{y}]_\Gamma=0,~[\beta\partial_{\mathbf{n}}\widetilde{y}]_{\Gamma}=v-\overline{u}, & ~\text{on}~ \Gamma\times(0,T),\\
 ~\widetilde{y}=0,&~\text{on}~ \partial\Omega\times(0,T),\\
 ~\widetilde{y}(0)=0,& ~\text{in} ~\Omega.
\end{array}\right.
\end{equation}
To further interpret the above condition, we introduce the following adjoint state equation:
\begin{equation}\label{s27}
\left\{\begin{array}{ll}
 ~-\overline{p}_{t}-\nabla\cdot(\beta\nabla \overline{p})  = \overline{y}-y_{d}, & ~\text{in}~ \Omega\setminus\Gamma\times(0,T),  \\
 ~[\overline{p}]_\Gamma=0,~[\beta\partial_{\mathbf{n}}\overline{p}]_{\Gamma}=0,& ~\text{on}~ \Gamma\times(0,T), \\
 ~\overline{p}=0,&~\text{on}~ \partial\Omega\times(0,T),\\
 ~\overline{p}(T)=0,& ~\text{in}~\Omega.
\end{array}\right.
\end{equation}
Choosing $\overline{p}$ as a test function in the weak formulation of the
equation (\ref{s26}) and integrating from $0$ to $T$, we obtain
\begin{eqnarray*}
-\int_{0}^{T}(\widetilde{y},\overline{p}_{t})dt+\int_{0}^{T}a( \widetilde{y},\overline{p})dt= \int_{0}^{T}\langle v-\overline{u},\overline{p}\rangle_{\Gamma} dt.
\end{eqnarray*}
On the other hand, we multiply both sides of (\ref{s27}) identically by $\widetilde{y}$ and integrate over the region:
\begin{eqnarray*}
-\int_{0}^{T}(\overline{p}_{t},\widetilde{y})dt+\int_{0}^{T}a(\overline{p}, \widetilde{y})dt=\int_{0}^{T}(\overline{y}-y_{d},\widetilde{y})dt.
\end{eqnarray*}
 With the above two formulas, we can get $\int_{0}^{T}\int_{\Omega}(\overline{y}-y_{d})\widetilde{y}dxdt= \int_{0}^{T}\int_{\Gamma} (v-\overline{u})\overline{p}ds dt$. The variational inequality reads
\begin{equation*}
J'(\overline{u})(v-\overline{u})=\int_{0}^{T}\int_{\Gamma}(\alpha \overline{u}+\overline{p}) (v-\overline{u})ds dt\geq 0,\quad \forall v\in U_{ad}.
\end{equation*}
\end{proof}
Using the optimality condition (\ref{s25}), we obtain the following regularity result.
\begin{lem} Let $(\overline{y},\overline{u},\overline{p})$ be the solution of the optimal control problem \textup{(\ref{s22})-(\ref{s24})}. If $y_{d}, f \in H^{1}(0,T;L^{2}(\Omega))$, $y_{0}\in H^{1}_{0}(\Omega)$ and $g\in L^{2}(0,T;H^{\frac{1}{2}}(\Gamma))$, we have
$$\overline{u}\in L^{2}(0,T;H^{\frac{1}{2}}(\Gamma))\cap H^{\frac{1}{4}}(0,T;L^{2}(\Gamma)) ,$$
$$\overline{y},\overline{p} \in L^{2}(0,T;H^{2}(\Omega^{+}\cup\Omega^{-}))\cap H^{1}(0,T;H^{1}(\Omega^{+}\cup\Omega^{-})).$$
\end{lem}
\begin{proof}
Note that $\overline{u} \in L^{2}(0,T;L^{2}(\Gamma))$ implies that $\overline{y} \in L^{2}(0,T;L^{2}(\Omega))$.  Since $\overline{y}-y_{d} \in L^{2}(0,T;L^{2}(\Omega))$, we conclude that $\overline{p} \in L^{2}(0,T;H^{2}(\Omega^{+}\cup\Omega^{-}))\cap H^{1}(0,T;L^{2}(\Omega))\cap L^{2}(0,T;H^{1}(\Omega))$ (see \cite{Sinha2005}),  hence $\overline{p}\mid_{\Gamma} \in L^{2}(0,T;H^{\frac{1}{2}}(\Gamma))\cap H^{\frac{1}{4}}(0,T;L^{2}(\Gamma))$. From (\ref{s25}) we obtain that  $\overline{u} \in L^{2}(0,T;H^{\frac{1}{2}}(\Gamma)) \cap H^{\frac{1}{4}}(0,T;L^{2}(\Gamma))$. Then applying Lemma 2.1, we have $\overline{y} \in L^{2}(0,T;H^{2}(\Omega^{+}\cup\Omega^{-}))\cap H^{1}(0,T;H^{1}(\Omega^{+}\cup\Omega^{-}))$.  This implies that $\overline{y}-y_{d} \in H^{1}(0,T;L^{2}(\Omega))$, so by Lemma 2.1 we have that $\overline{p} \in L^{2}(0,T;H^{2}(\Omega^{+}\cup\Omega^{-}))\cap H^{1}(0,T;H^{1}(\Omega^{+}\cup\Omega^{-}))$.
\end{proof}
We note that for optimal control problems posed on the domain $\Omega$ with smooth boundary we can get higher regularity.
\begin{lem} Let $(\overline{y},\overline{u},\overline{p})$ be the solution of the optimal control problem \textup{(\ref{s22})-(\ref{s24})}. If $y_{d}, f \in L^{q}(0,T;L^{q}(\Omega)) \cap H^{1}(0,T;L^{2}(\Omega))$, $g\in L^{q}(0,T;W^{1-\frac{1}{q},q}(\Gamma))\cap W^{\frac{1}{2}-\frac{1}{2q},2}(0,T;L^{q}(\Gamma)) \cap L^{2}(0,T;H^{\frac{1}{2}}(\Gamma))$, $y_{0}\in B^{2-\frac{2}{q}}_{q,q}(\Omega^{+}\cup\Omega^{-})$ for $2< q< 3$ and the given data  satisfy required compatibility condition, we have
$$\overline{u}\in L^{q}(0,T;W^{1-\frac{1}{q},q}(\Gamma))\cap W^{\frac{1}{2}-\frac{1}{2q},q}(0,T;L^{q}(\Gamma)),$$
$$\overline{y},\overline{p} \in L^{q}(0,T;W^{2,q}(\Omega^{+}\cup \Omega^{-}))\cap W^{1,q}(0,T;L^{q}(\Omega)).$$
\end{lem}
\begin{proof}
 From Lemma 2.4 we already have $y \in  L^{2}(0,T;H^{2}(\Omega^{+}\cup\Omega^{-}))\cap H^{1}(0,T; H^{1}(\Omega^{+}\cup\Omega^{-}))$. In particular, this implies $\overline{y} \in L^{q}(0,T;L^{q}(\Omega))$. If  $y_{d} \in L^{q}(0,T;L^{q}(\Omega))$, we have the improved regularity $\overline{p} \in L^{q}(0,T;W^{2,q}(\Omega^{+}\cup \Omega^{-}))\cap W^{1,q}(0,T;L^{q}(\Omega))$ from Lemma 2.2, hence $\overline{p}|_{\Gamma} \in L^{q}(0,T;W^{2-\frac{1}{q},q}(\Gamma))\cap W^{1-\frac{1}{2q},q}(0,T;L^{q}(\Gamma))$ (see \cite{Weidemaier2002}). From (\ref{s25}) we conclude that $\overline{u} \in L^{q}(0,T;W^{1-\frac{1}{q},q}(\Gamma))\cap W^{\frac{1}{2}-\frac{1}{2q},q}(0,T;L^{q}(\Gamma))$. Then applying Lemma 2.2, we obtain $\overline{y} \in L^{q}(0,T;W^{2,q}(\Omega^{+}\cup \Omega^{-}))\cap W^{1,q}(0,T;L^{q}(\Omega))$.
\end{proof}
\begin{rem}
Assume that $\Omega$ is a bounded domain with smooth boundary. According to Lemma 2.5, we obtain $\overline{y},\overline{p} \in L^{q}(0,T;W^{2,q}(\Omega^{+}\cup\Omega^{-}))\cap W^{1,q}(0,T;L^{q}(\Omega))$. The trace theorem then implies $\overline{y}, \overline{p}|_{\Gamma} \in L^{q}(0,T;W^{2-\frac{1}{q},q}(\Gamma))$ (see \cite{Weidemaier2002}). Using the Sobolev embedding $W^{2-\frac{1}{q},q}(\Gamma) \hookrightarrow W^{1,\infty}(\Gamma)$ for $q>2$, we obtain $\overline{y}, \overline{p}|_{\Gamma} \in L^{q}(0,T;W^{1,\infty}(\Gamma))$, hence $\overline{y},\overline{p} |_{\Gamma}\in L^{2}(0,T;W^{1,\infty}(\Gamma))$. This regularity ensures the regularity assumption of Theorem 3.1 and Theorem 4.1.
\end{rem}
\section{\bf The Discretization of the Optimal Control Problem}
In this section, we consider the fully discrete approximation of the control problem (\ref{s22}). For the discretization of the state equation, we use the stable generalized finite element method (SGFEM) for the spatial discretization and the backward Euler scheme for the temporal discretization. The discretization of the control variable is obtained by the projection of the discretized adjoint state on the set of admissible controls, the so-called variational discretization (see \cite{Hinze2005}).
\subsection{ The stable generalized finite element method}
Let $\mathcal T_{h}=\{K\}$ denote a uniform triangulation of $\Omega$ with mesh size $h$. Denote $\{P_{i}\}_{i\in I_{h}}$ to be the set of finite
element nodes associated with the mesh $\mathcal T_{h} $, where $I_{h}$ is the index set of the nodes. For every $i\in I_{h}$, we consider the standard linear finite element basis function $\phi_{i}$. The approximate subspace of the GFEM is defined as:
\begin{equation*}
\mathbb{S}_{h} = \mathbb{S}_{FEM} \oplus \mathbb{S}_{ENR},
\end{equation*}
and
$$\mathbb{S}_{FEM}  = \text{span} \{ \phi_i : i \in I_h \},~\mathbb{S}_{ENR} = \text{span} \{  \phi_i\Pi_{i} : i \in I_{enr} \subset I_h\},$$
where $I_{enr}=\{i \in I_h: P_{i} \in K ~\text{where}~ K\cap\Gamma\neq\phi\}$ denotes the index set of enrichment nodes (see Fig. 2). The enrichment function $\Pi_{i}$ is generally based on the absolute value of the level set function \cite{Fries2010}
$$D(P) = |\varphi(P)|,~ \varphi(\cdot) ~\text{is a level set function}$$
or the distance function \cite{Kergrene2016,Sukumar2001}
$$D(P) = \text{dist}(P, \Gamma), ~\text{dist}(P, \Gamma) \text{ is the distance of point } P \text{ to the interface } \Gamma.$$
Unfortunately, the condition number of the stiffness matrices can be very large when the GFEM is applied to the interface problem \cite{Babuka2012,Babuka2017,Kergrene2016,Zhang2014,Zhu2020}. This is mainly caused by almost linear dependence between the FE functions and added special functions. To address the bad conditioning of the GFEM, a stable GFEM (SGFEM) was proposed in \cite{Babuka2017,Zhang2019,Gupta2015,Kergrene2016,Zhang2014,Zhang2016,Gupta2013,Zhang2020}. The main idea is to modify the enrichment space by subtracting the interpolation of the enrichment function. The approximate subspace of the SGFEM is given by
\begin{equation*}
\mathbb{S}_{h} = \mathbb{S}_{FEM} \oplus \mathbb{S}_{ENR} ~\text{and}~ \mathbb{S}_{ENR} = \text{span} \{  \phi_i(D-\mathcal{I}_h D) : i \in I_{enr} \subset I_h\},
\end{equation*}
where $\mathcal{I}_h w$ is the finite element interpolant of $w$. In \cite{Zhu2020}, the author designed an SGFEM with a one-sided distance function that is simpler than the standard distance function or its other versions:
$$
\tilde{D}(x):=\left\{\begin{array}{llll}
D(x), &x  \in \Omega^{+}, \\
0, & x \in  \Omega^{-} .
\end{array}\right.
$$
To reduce the computational cost of the enrichment function, we use the enrichment function proposed in \cite{Zhu2020}.
\begin{figure}
    \centering
    \includegraphics[width=0.4\linewidth]{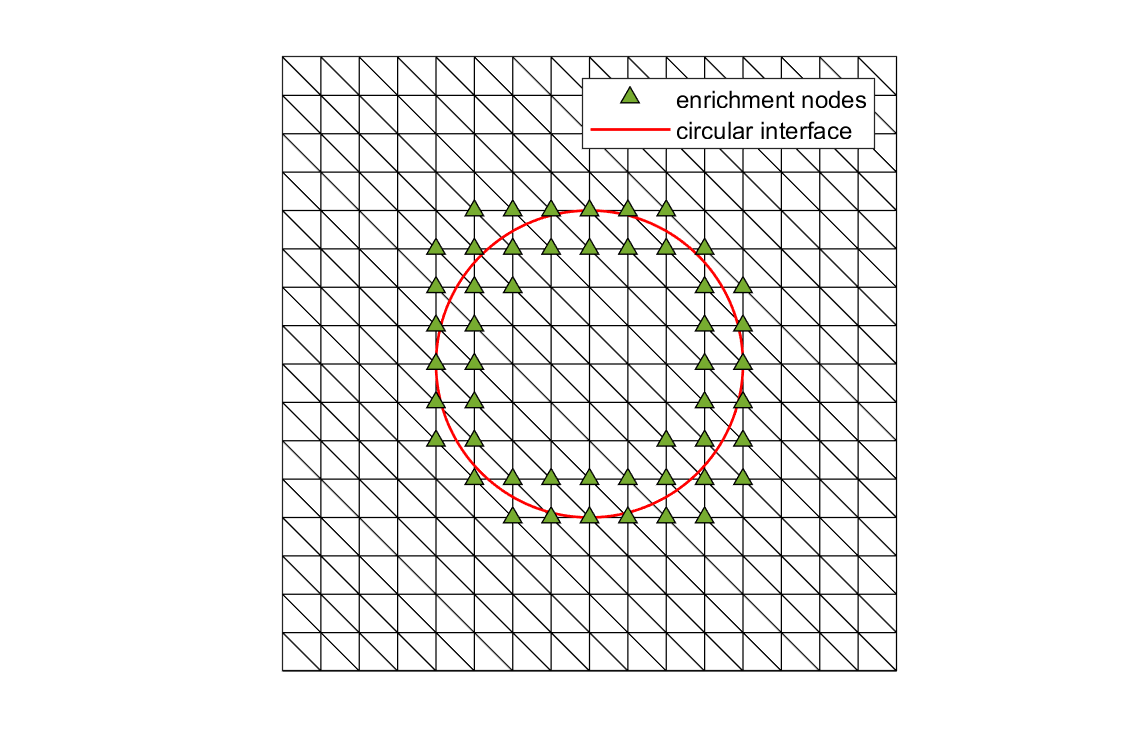}
    \caption*{Fig. 2. An illustration of the enrichment nodes $P_i, i\in I_{enr}$.}
\end{figure}
For every $t \in [0, T]$ and $w$, we define the elliptic projection $R_{h}w \in \mathbb{S}_{h}$ on $\varepsilon(\Omega)$:
$$a(w-R_{h}w,v_{h})=0,\quad \forall v_{h} \in \mathbb{S}_{h},$$
with $\displaystyle \int_{\Omega}\big(w(t)-R_{h}w(t)\big) dx =0 $.
\begin{lem}\textup{\cite{Zhu2020}}
 Let $R_{h}$ be the elliptic projection operator defined above. Then there exists $C > 0$ independent of $h$ such that
\begin{equation*}
\begin{aligned}
&\|w-R_{h}w\|_{\varepsilon,\Omega}\leq Ch\|w\|_{\mathbb{X}},\\
&\|w-R_{h}w\|_{0,\Omega}\leq Ch^{2}\|w\|_{\mathbb{X}},
\end{aligned}
\end{equation*}
where $\mathbb{X}:=\{w\in \varepsilon(\Omega):w|_{\Omega^{-}}\in H^{2}(\Omega^{-}),w|_{\Omega^{+}}\in H^{2}(\Omega^{+}), [w]_{\Gamma}=0~\text{and}~ ~\|\nabla w \|_{\infty,\Gamma}<\infty \}$ with norm $\|w\|_{\mathbb{X}}:=\|w\|_{2,\Omega^{+}}+\|w\|_{2,\Omega^{-}}+\|\nabla w \|_{\infty,\Gamma}$.
\end{lem}
\subsection{Fully discrete approximation of optimal control problems}
We consider the fully discrete approximation for the state equation (\ref{s21}) by
using the stable generalized finite element method and the backward Euler method. We consider a uniform partitioning of the time interval $[0, T]$ by the points $0=t_{0}<t_{1}<\cdots<t_{M-1}<t_{M}=T$ with $ t_{n} =n\Delta t,$ $\Delta t=T/M$ being the time step. Let $I_{n}=(t_{n-1},t_{n}]$ be the nth subinterval and $v^{n}$ denote the value of a function $v(x,t)$ at $t_{n}$. \par
For a given control  $u \in U_{ad}$ the fully discrete approximation of the state equation (\ref{s21}) is defined as follows: Find a state $Y_{h}^{n}(u)\in \mathbb{S}_{h}\cap H_{0}^{1}(\Omega)$ such that
\begin{equation}\label{s32}
\begin{aligned}
\Big(\frac{Y_{h}^{n}(u)-Y_{h}^{n-1}(u)}{\Delta t},w_{h}\Big)+a(Y_{h}^{n}(u),w_{h})=\frac{1}{\Delta t}\int_{I_{n}}(f,w_{h})dt+\frac{1}{\Delta t}\int_{I_{n}}\langle g+u,w_{h}\rangle_{\Gamma}dt,&\\
\forall w_{h} \in \mathbb{S}_{h}\cap H_{0}^{1}(\Omega), ~n=1,2,\cdots,M, ~\text{with}~ Y_{h}^{0}(u)=R_{h}y_{0}.&\\
\end{aligned}
\end{equation}
In the following we denote $Y_h(u)$ as the fully discrete finite element approximation of $y(u)$, i.e., $Y_h(u)|_{I_{n}}=Y_{h}^{n}(u), ~n=1,2,\cdots,M$.
For the error analysis derived later, we will need the following interpolant $\overline{P}^{n}_{k}$ defined by
\begin{equation*}
\overline{P}^{n}_{k}v=\frac{1}{\Delta t}\int_{I_{n}}v(\cdot,t)dt,\quad n=1,2,\ldots,M ~ \text{and}~\overline{P}^{0}_{k}v =v^{0}
\end{equation*}
and the interpolant $\overline{R}^{n}_{h}$ defined by
\begin{equation*}
\overline{R}^{n}_{h}v=\frac{1}{\Delta t}\int_{I_{n}}R_{h}v(\cdot,t)dt,\quad n=1,2,\ldots,M ~ \text{and}~\overline{R}^{0}_{h}v =R_{h}v^{0}.
\end{equation*}
It is easy to show that
\begin{equation}\label{s4666}
\Big(\sum\limits_{n=1}^{M}\Delta t\| v^{n}-\overline{P}_{k}^{n}v\|_{0,\Omega}^{2}\Big)^{\frac{1}{2}}\leq C\Delta t\| v_{t}\|_{L^{2}(0,T;L^{2}(\Omega))}
\end{equation}
and
\begin{equation}\label{s4667}
\Big(\sum\limits_{n=1}^{M}\Delta t\| \overline{P}_{k}^{n}v-\overline{R}_{h}^{n}v\|_{0,\Omega}^{2}\Big)^{\frac{1}{2}}\leq Ch^{2}\| v\|_{L^{2}(0,T;\mathbb{X})}.
\end{equation}
Consider the following auxiliary problem: Find $z_{h}^{n} \in \mathbb{S}_{h} \cap H_{0}^{1}(\Omega)$ such that
\begin{equation}\label{s37}
\begin{aligned}
\Big(\frac{z_{h}^{n-1}-z_{h}^{n}}{\Delta t},w_{h}\Big)+a(z_{h}^{n-1},w_{h})=\frac{1}{\Delta t}\int_{I_{n}}(\varphi,w_{h})dt,\quad \forall w_{h} \in \mathbb{S}_{h}\cap H_{0}^{1}(\Omega),&\\
 n=1,2,\cdots,M,&
\end{aligned}
\end{equation}
with $z^{M}_{h}=0$. In the following lemma we provide a stability estimate for the solution
of (\ref{s37}).
\begin{lem} For given $\varphi \in L^{2}(0,T;L^{2}(\Omega))$, let $z_{h}^{n}\in \mathbb{S}_{h}\cap H_{0}^{1}(\Omega)$ be the solution of equation \textup{(\ref{s37})}. Then it holds that
\begin{equation*}
\| z_{h}^{0} \|_{\varepsilon,\Omega}^{2}+\sum\limits_{n=1}^{M} \Delta t^{-1}\|z_{h}^{n}-z_{h}^{n-1}\|_{0,\Omega}^{2}\leq C\sum\limits_{n=1}^{M}\int_{I_{n}}\| \varphi\|_{0,\Omega}^{2}dt.
\end{equation*}
\end{lem}
\begin{proof}
Setting $w_{h}=z_{h}^{n-1}-z_{h}^{n}$ in (\ref{s37}) and using the Cauchy-Schwarz inequality and Young's inequality, we obtain
\begin{equation*}
\begin{aligned}
&\frac{1}{\Delta t}\|z_{h}^{n}-z_{h}^{n-1}\|_{0,\Omega}^{2}
+\frac{1}{2}\|z_{h}^{n-1}\|_{\varepsilon,\Omega}^{2}-\frac{1}{2}\|z_{h}^{n}\|_{\varepsilon,\Omega}^{2}+\frac{1}{2}\|z_{h}^{n}-z_{h}^{n-1}\|_{\varepsilon,\Omega}^{2}\\
&\leq \frac{1}{2}\int_{I_{n}}\|\varphi\|_{0,\Omega}^{2}dt+\frac{1}{2\Delta t}\|z_{h}^{n}-z_{h}^{n-1}\|_{0,\Omega}^{2}.
\end{aligned}
\end{equation*}
Summation of the equations for $n = 1, 2, \ldots , M$ leads to
\begin{equation*}
\sum\limits_{n=1}^{M}\Delta t^{-1}\|z_{h}^{n}-z_{h}^{n-1}\|_{0,\Omega}^{2}+\|z_{h}^{0}\|^{2}_{\varepsilon,\Omega}
\leq C\sum\limits_{n=1}^{M}\int_{I_{n}}\|\varphi\|_{0,\Omega}^{2}dt.
\end{equation*}
\end{proof}
\begin{thm} For $u\in U_{ad}$, let $y(u)$ and $Y_{h}(u)$ be the solutions of equations \textup{(\ref{s21})} and \textup{(\ref{s32})}, respectively.  Then we have the following a priori error estimate:
\begin{equation*}
\|y(u)-Y_{h}(u)\|_{L^{2}(0,T;L^{2}(\Omega))}\leq C\Big(h^{2}\|y_{0}(u)\|_{\mathbb{X}}+\Delta t \|y_{t}(u)\|_{L^{2}(0,T;L^{2}(\Omega))}+h^{2}\|y(u)\|_{L^{2}(0,T;\mathbb{X})}\Big).
\end{equation*}
\end{thm}
\begin{proof}
We split the error
\begin{equation}\label{s38}
\begin{aligned}
\|y(u)-Y_{h}(u)\|_{L^{2}(0,T;L^{2}(\Omega))}&\leq \Big(\sum\limits_{n=1}^{M}\int_{I_{n}}\|y(u)-\overline{R}_{h}^{n}y(u)\|_{0,\Omega}^{2}dt\Big)^{\frac{1}{2}}\\
&\quad+\Big(\sum\limits_{n=1}^{M}\int_{I_{n}}\|\overline{R}_{h}^{n}y(u)-Y_{h}^{n}(u)\|_{0,\Omega}^{2}dt\Big)^{\frac{1}{2}}
\end{aligned}
\end{equation}
and estimate both terms on the right-hand side separately. For the first term, using (\ref{s4667}) we obtain
\begin{equation}\label{s39}
\Big(\sum\limits_{n=1}^{M}\int_{I_{n}}\|y(u)-\overline{R}_{h}^{n}y(u)\|_{0,\Omega}^{2}dt\Big)^{\frac{1}{2}}\leq C\Big(\Delta t \|y_{t}(u)\|_{L^{2}(0,T;L^{2}(\Omega))}+h^{2}\|y(u)\|_{L^{2}(0,T;\mathbb{X})}\Big).
\end{equation}
Let $z_{h}^{n}$ be the solution of problem (\ref{s37}) with $\varphi=\overline{R}_{h}^{n}y(u)-Y_{h}^{n}(u)$. Choose $w_{h}=\overline{R}_{h}^{n}y(u)-Y_{h}^{n}(u)$ in (\ref{s37}) and summing in time, we obtain
\begin{align*}
&\sum\limits_{n=1}^{M}\int_{I_{n}}\|\overline{R}_{h}^{n}y(u)-Y_{h}^{n}(u)\|_{0,\Omega}^{2}dt\\
&=\sum\limits_{n=1}^{M}\big(z_{h}^{n-1}-z_{h}^{n},\overline{R}_{h}^{n}y(u)-Y_{h}^{n}(u)\big)+\sum\limits_{n=1}^{M}\Delta ta\big(z_{h}^{n-1},\overline{R}_{h}^{n}y(u)-Y_{h}^{n}(u)\big)\\
&=\sum\limits_{n=1}^{M}\big(z_{h}^{n-1}-z_{h}^{n},\overline{R}_{h}^{n}y(u)\big)+\sum\limits_{n=1}^{M}\Delta ta\big(z_{h}^{n-1},\overline{R}_{h}^{n}y(u)\big)\\
&\quad-\sum\limits_{n=1}^{M}\big(z_{h}^{n-1}-z_{h}^{n},Y_{h}^{n}(u)\big)-\sum\limits_{n=1}^{M}\Delta ta\big(z_{h}^{n-1},Y_{h}^{n}(u)\big)\\
&=\sum\limits_{n=1}^{M}\big(z_{h}^{n-1}-z_{h}^{n},\overline{R}_{h}^{n}y(u)\big)+\sum\limits_{n=1}^{M}\Delta ta\big(z_{h}^{n-1},\overline{R}_{h}^{n}y(u)\big)-(z_{h}^{0},Y_{h}^{0}(u))\\
&\quad-\sum\limits_{n=1}^{M}\big(z_{h}^{n-1},Y_{h}^{n}(u)-Y_{h}^{n-1}(u)\big)-\sum\limits_{n=1}^{M}\Delta ta\big(z_{h}^{n-1},Y_{h}^{n}(u)\big)\\
&=\sum\limits_{n=1}^{M}\big(z_{h}^{n-1}-z_{h}^{n},\overline{R}_{h}^{n}y(u)\big)+\sum\limits_{n=1}^{M}\Delta ta\big(z_{h}^{n-1},\overline{R}_{h}^{n}y(u)\big)-(z_{h}^{0},Y_{h}^{0}(u))\\
&\quad-\sum\limits_{n=1}^{M}\int_{I_{n}}(f,z_{h}^{n-1})dt-\sum\limits_{n=1}^{M}\int_{I_{n}}\langle g+u,z_{h}^{n-1}\rangle_{\Gamma}dt,
\end{align*}
where the last line follows from (\ref{s32}) and $z_{h}^{M}=0$. From (\ref{s21}) and the definition of $ \overline{R}_{h}^{n}$ this becomes
\begin{align}\label{s310}
&\sum\limits_{n=1}^{M}\int_{I_{n}}\|\overline{R}_{h}^{n}y(u)-Y_{h}^{n}(u)\|_{0,\Omega}^{2}dt\\
&=\sum\limits_{n=1}^{M}\big(z_{h}^{n-1}-z_{h}^{n},\overline{R}_{h}^{n}y(u)\big)+\sum\limits_{n=1}^{M}\Delta ta\big(z_{h}^{n-1},\overline{R}_{h}^{n}y(u)\big)-(z_{h}^{0},R_{h}y_{0})\\
&\quad-\sum\limits_{n=1}^{M}\big(z_{h}^{n-1},y^{n}(u)-y^{n-1}(u)\big)-\sum\limits_{n=1}^{M}\Delta ta\big(z_{h}^{n-1},P_{k}^{n}y(u)\big)\\
&=\sum\limits_{n=1}^{M}\big(z_{h}^{n-1}-z_{h}^{n},\overline{R}_{h}^{n}y(u)-y^{n}(u)\big)+(y_{0}-R_{h}y_{0},z_{h}^{0})\\
&=E_{1}+E_{2}.
\end{align}
For the term $E_{1}$, using (\ref{s4666})-(\ref{s4667}) and Lemma 3.2 leads to
\begin{equation*}
\begin{aligned}
E_{1}&\leq\Big(\sum\limits_{n=1}^{M}\Delta t^{-1}\|z_{h}^{n-1}-z_{h}^{n}\|_{0,\Omega}^{2}\Big)^{\frac{1}{2}}\Big(\sum\limits_{n=1}^{M}\Delta t\|\overline{R}_{h}^{n}y(u)-y^{n}(u)\|_{0,\Omega}^{2}\Big)^{\frac{1}{2}}\\
&\leq C\Big(\sum\limits_{n=1}^{M}\int_{I_{n}}\|\overline{R}_{h}^{n}y(u)-Y_{h}^{n}(u)\|_{0,\Omega}^{2}dt\Big)^{\frac{1}{2}}\Big(\Delta t \|y_{t}(u)\|_{L^{2}(0,T;L^{2}(\Omega))}+h^{2}\|y(u)\|_{L^{2}(0,T;\mathbb{X})}\Big).
\end{aligned}
\end{equation*}
Using Lemma 3.1 and Lemma 3.2, the term $E_{2}$ is estimated as
\begin{equation*}
E_{2}\leq\|y_{0}-R_{h}y_{0}\|_{0,\Omega}\|z_{h}^{0}\|_{0,\Omega}\leq Ch^{2}\|y_{0}(u)\|_{\mathbb{X}}\Big(\sum\limits_{n=1}^{M}\int_{I_{n}}\|\overline{R}_{h}^{n}y(u)-Y_{h}^{n}(u)\|_{0,\Omega}^{2}dt\Big)^{\frac{1}{2}}.
\end{equation*}
Inserting the estimates for $E_{1}$ and $E_{2}$ into (\ref{s310}) yields
\begin{equation}\label{s313}
\begin{aligned}
&\Big(\sum\limits_{n=1}^{M}\int_{I_{n}}\|\overline{R}_{h}^{n}y(u)-Y_{h}^{n}(u)\|_{0,\Omega}^{2}dt\Big)^{\frac{1}{2}}\\
&\leq C\Big(h^{2}\|y_{0}(u)\|_{\mathbb{X}}+\Delta t \|y_{t}(u)\|_{L^{2}(0,T;L^{2}(\Omega))}+h^{2}\|y(u)\|_{L^{2}(0,T;\mathbb{X})}\Big).
\end{aligned}
\end{equation}
By inserting the estimates (\ref{s39}) and (\ref{s313}) into
(\ref{s38}), the proof is completed.
\end{proof}
Based on the discretization of the state equation (\ref{s32}), for the discretization of the control variable we use the variational discretization approach proposed by Hinze in  \cite{Hinze2005}. The fully discrete approximation scheme of the optimal control problem is given as follows:
\begin{equation}\label{s33}
\min\limits_{u\in U_{ad}} J(Y_{h}(u),u)=\frac{1}{2}\sum\limits_{n=1}^{M}\int_{I_{n}} \| Y_{h}^{n}(u)-y_{d}\|_{L^{2}(\Omega)}^{2} dt+\frac{\alpha}{2}\sum\limits_{n=1}^{M}\int_{I_{n}} \| u\|_{L^{2}(\Gamma)}^{2}dt,
\end{equation}
subject to
\begin{equation}\label{s34}
\begin{aligned}
\Big(\frac{Y_{h}^{n}(u)-Y_{h}^{n-1}(u)}{\Delta t},w_{h}\Big)+a(Y_{h}^{n}(u),w_{h})=\frac{1}{\Delta t}\int_{I_{n}}(f,w_{h})dt+\frac{1}{\Delta t}\int_{I_{n}}\langle g+u,w_{h}\rangle_{\Gamma}dt,&\\
\forall w_{h} \in \mathbb{S}_{h}\cap H_{0}^{1}(\Omega),~n=1,2,\cdots,M, ~\text{with}~ Y_{h}^{0}(u)=R_{h}y_{0}.&\\
\end{aligned}
\end{equation}
This problem admits a unique solution $(Y_{h}^{n},U_{h})\in (\mathbb{S}_{h} \cap H_{0}^{1}(\Omega))\times U_{ad}$. Moreover, we have the following first order optimality conditions: There exists a unique discrete adjoint state $P_{h}^{n} \in \mathbb{S}_{h}\cap H_{0}^{1}(\Omega)$ such that
\begin{equation}\label{s35}
\begin{aligned}
\Big(\frac{P_{h}^{n-1}-P_{h}^{n}}{\Delta t},w_{h}\Big)+a(P_{h}^{n-1},w_{h})=\frac{1}{\Delta t}\int_{I_{n}}(Y_{h}^{n}-y_{d},w_{h})dt,&\\
\forall w_{h} \in \mathbb{S}_{h}\cap H_{0}^{1}(\Omega), ~n=M, M-1,\cdots,1, ~\text{with}~ P_{h}^{M}=0,&\\
\end{aligned}
\end{equation}
and
\begin{equation}\label{s36}
\sum\limits_{n=1}^{M}\int_{I_{n}}\langle\alpha U_{h}+P_{h}^{n-1},v-U_{h}\rangle_{\Gamma}dt\geq 0,\quad \forall v\in U_{ad}.
\end{equation}
In the following we denote $Y_h|_{I_{n}}=Y_{h}^{n}$ and $P_h\mid_{I_{n}}=P_{h}^{n-1}$, for $ n=1,2,\cdots,M$, where $Y_{h}^{n}$ and $P_{h}^{n-1}$ are solutions of equations (\ref{s34}) and  (\ref{s35}), respectively. 
\section{Error Analysis of Optimal Control Problems}
For the subsequent analysis, it is convenient to introduce the following two auxiliary problems. For a given $u\in U_{ad}$, let $y:=y(u)$ be the solution of the state equation (\ref{s21}). For given $y$, find $p(y)\in L^{2}(0,T;L^{2}(\Omega))$ and $P_{h}^{n}(y) \in \mathbb{S}_{h}\cap H_{0}^{1}(\Omega)$ satisfying
\begin{equation}\label{s41}
\begin{aligned}
-(p_{t}(y),w)+a(p(y),w)&=(y-y_{d},w), \quad \forall  w\in H^{1}_{0}(\Omega), ~t\in (0,T),\\
p(y)(x,T)&=0,\quad x\in \Omega,
\end{aligned}
\end{equation}
and
\begin{equation}\label{s42}
\begin{aligned}
\Big(\frac{P_{h}^{n-1}(y)-P_{h}^{n}(y)}{\Delta t},w_{h}\Big)+a(P_{h}^{n-1}(y),w_{h})=\frac{1}{\Delta t}\int_{I_{n}}(y-y_{d},w_{h})dt,&\\
\forall w_{h} \in \mathbb{S}_{h}\cap H_{0}^{1}(\Omega), n=M, M-1,\cdots,1, ~\text{with}~ P_{h}^{M}(y)=0,&\\
\end{aligned}
\end{equation}
respectively. Let $P_{h}(y)|_{I_{n}}=P_{h}^{n-1}(y),~ n=1,2,\cdots,M$.
\begin{thm} Assume that $p(y)$ and $P_{h}(y)$ are the solutions of equations \textup{(\ref{s41})} and \textup{(\ref{s42})}, respectively.  Then we have the following an a priori error estimate:
\begin{equation*}
\begin{aligned}
\Big(\sum\limits_{n=1}^{M}\int_{I_{n}}\|\overline{R}_{h}^{n}p(y)-P_{h}^{n-1}(y)\|_{0,\Omega}^{2}dt\Big)^{\frac{1}{2}}&\leq C\Big(\Delta t \|p_{t}(y)\|_{L^{2}(0,T;L^{2}(\Omega))}+h^{2}\|p(y)\|_{L^{2}(0,T;\mathbb{X})}\Big),\\
\Big(\sum\limits_{n=1}^{M}\int_{I_{n}}\|\overline{R}_{h}^{n}p(y)-P_{h}^{n-1}(y)\|_{\varepsilon,\Omega}^{2}dt\Big)^{\frac{1}{2}}&\leq C\Big(\Delta t^{\frac{1}{2}}\|p_{t}(y)\|_{L^{2}(0,T;L^{2}(\Omega))}^{\frac{1}{2}}+h\|p(y)\|_{L^{2}(0,T;\mathbb{X})}^{\frac{1}{2}}\Big)\\
&\quad\times\Big(\|p_{t}(y)\|_{L^{2}(0,T;L^{2}(\Omega))}^{\frac{1}{2}}+\|y-y_{d}\|_{L^{2}(0,T;L^{2}(\Omega))}^{\frac{1}{2}}\Big)
.\\
\end{aligned}
\end{equation*}
\end{thm}
\begin{proof}
Similar to the proof for Theorem 3.1, we have
\begin{equation}\label{s46}
\Big(\sum\limits_{n=1}^{M}\int_{I_{n}}\|\overline{R}_{h}^{n}p(y)-P_{h}^{n-1}(y)\|_{0,\Omega}^{2}dt\Big)^{\frac{1}{2}}
\leq C\Big(\Delta t \|p_{t}(y)\|_{L^{2}(0,T;L^{2}(\Omega))}+h^{2}\|p(y)\|_{L^{2}(0,T;\mathbb{X})}\Big).
\end{equation}
Taking $w=\overline{R}_{h}^{n}p(y)-P_{h}^{n-1}(y)$ in (\ref{s41}) and $w_{h}=\overline{R}_{h}^{n}p(y)-P_{h}^{n-1}(y)$ in (\ref{s42}), we obtain
\begin{equation*}
\begin{aligned}
\big(p^{n-1}(y)-p^{n}(y),\overline{R}_{h}^{n}p(y)-P_{h}^{n-1}(y)\big)-\big(P_{h}^{n-1}(y)-P_{h}^{n}(y),\overline{R}_{h}^{n}p(y)-P_{h}^{n-1}(y)\big)\\
+\Delta ta\big(\overline{P}_{k}^{n}p(y)-P_{h}^{n-1}(y),\overline{R}_{h}^{n}p(y)-P_{h}^{n-1}(y)\big)=0.
\end{aligned}
\end{equation*}
With the Cauchy-Schwarz inequality and the definition of $\overline{R}_{h}^{n}$, we obtain
\begin{equation*}
\begin{aligned}
&\Delta ta\big(\overline{R}_{h}^{n}p(y)-P_{h}^{n-1}(y),\overline{R}_{h}^{n}p(y)-P_{h}^{n-1}(y)\big)\\
&\leq \|p^{n-1}(y)-p^{n}(y)\|_{0,\Omega}\|\overline{R}^{n}_{h}p(y)-P_{h}^{n-1}(y)\|_{0,\Omega}\\
&\quad+\|P_{h}^{n-1}(y)-P_{h}^{n}(y)\|_{0,\Omega}\|\overline{R}_{h}^{n}p(y)-P_{h}^{n-1}(y)\|_{0,\Omega}.
\end{aligned}
\end{equation*}
Summation of the above inequalities for  $n=1,2,\cdots,M$ and application of the Cauchy-Schwarz inequality admit
\begin{equation}\label{s43}
\begin{aligned}
&\sum\limits_{n=1}^{M}\Delta ta\big(\overline{R}_{h}^{n}p(y)-P_{h}^{n-1}(y),\overline{R}_{h}^{n}p(y)-P_{h}^{n-1}(y)\big)\\
&\leq \Big(\sum\limits_{n=1}^{M}\Delta t^{-1}\|p^{n-1}(y)-p^{n}(y)\|_{0,\Omega}^{2}\Big)^{\frac{1}{2}}\Big(\sum\limits_{n=1}^{M}\Delta t\|\overline{R}_{h}^{n}p(y)-P_{h}^{n-1}(y)\|_{0,\Omega}^{2}\Big)^{\frac{1}{2}}\\
&\quad+\Big(\sum\limits_{n=1}^{M}\Delta t^{-1}\|P_{h}^{n-1}(y)-P_{h}^{n}(y)\|_{0,\Omega}^{2}\Big)^{\frac{1}{2}}\Big(\sum\limits_{n=1}^{M}\Delta t\|\overline{R}_{h}^{n}p(y)-P_{h}^{n-1}(y)\|_{0,\Omega}^{2}\Big)^{\frac{1}{2}}.
\end{aligned}
\end{equation}
Simple calculation leads to
\begin{equation}\label{s44}
\begin{aligned}
\sum\limits_{n=1}^{M}\Delta t^{-1}\|p^{n-1}(y)-p^{n}(y)\|_{0,\Omega}^{2}\leq\sum\limits_{n=1}^{M}\int_{I_{n}}\|p_{t}(y)\|_{0,\Omega}^{2}dt=\|p_{t}(y)\|^{2}_{L^{2}(0,T;L^{2}(\Omega))}.
\end{aligned}
\end{equation}
Taking $w_{h}=P^{n-1}_{h}(y)-P_{h}^{n}(y)$ in (\ref{s42}), we can get the following result by the similar method in the proof of Lemma 3.2.
\begin{equation}\label{s45}
\sum\limits_{n=1}^{M}\Delta t^{-1}\|P^{n-1}_{h}(y)-P_{h}^{n}(y)\|_{0,\Omega}^{2}\leq C\sum\limits_{n=1}^{M}\int_{I_{n}}\|y-y_{d}\|_{0,\Omega}^{2}dt.
\end{equation}
Then from (\ref{s43}), (\ref{s44}), (\ref{s45}) and (\ref{s46}) we have
\begin{equation*}
\begin{aligned}
&\sum\limits_{n=1}^{M}\Delta ta\big(\overline{R}_{h}^{n}p(y)-P_{h}^{n-1}(y),\overline{R}_{h}^{n}p(y)-P_{h}^{n-1}(y)\big)\\
&\leq \Big(\|p_{t}(y)\|_{L^{2}(0,T;L^{2}(\Omega))}+\|y-y_{d}\|_{L^{2}(0,T;L^{2}(\Omega))}\Big)\\
&\quad \times\Big(\Delta t \|p_{t}(y)\|_{L^{2}(0,T;L^{2}(\Omega))}+h^{2}\|p(y)\|_{L^{2}(0,T;\mathbb{X})}\Big).
\end{aligned}
\end{equation*}
This completes the proof.
\end{proof}
\begin{thm} Let $p(y)$ and $P_{h}(y)$ be the solutions of equations \textup{(\ref{s41})} and \textup{(\ref{s42})}, respectively.  Then we have the following an a priori error estimate:
\begin{equation*}
\begin{aligned}
\Big(\sum\limits_{n=1}^{M}\int_{I_{n}}\|p(y)-P_{h}^{n-1}(y)\|_{0,\Gamma}^{2}dt\Big)^{\frac{1}{2}}&\leq C(\Delta t^{3/4}+h^{3/2}+\Delta t^{1/2}h^{1/2}+\Delta t^{1/4}h).
\end{aligned}
\end{equation*}
\end{thm}
\begin{proof}
Using the Cauchy-Schwarz inequality and trace estimate, we obtain
\begin{align*}
&\sum\limits_{n=1}^{M}\int_{I_{n}}\|p(y)-P_{h}^{n-1}(y)\|_{0,\Gamma}^{2}dt\\
&\leq C\Big(\sum\limits_{n=1}^{M}\int_{I_{n}}\|p(y)-P_{h}^{n-1}(y)\|_{0,\partial\Omega^{-}}^{2}dt+\sum\limits_{n=1}^{M}\int_{I_{n}}\|p(y)-P_{h}^{n-1}(y)\|_{0,\partial\Omega^{+}}^{2}dt\Big)\\
&\leq C\Big(\sum\limits_{n=1}^{M}\int_{I_{n}}\|p(y)-P_{h}^{n-1}(y)\|_{0,\Omega^{-}}^{2}dt\Big)^{\frac{1}{2}}\Big(\sum\limits_{n=1}^{M}\int_{I_{n}}\|p(y)-P_{h}^{n-1}(y)\|_{1,\Omega^{-}}^{2}dt\Big)^{\frac{1}{2}}\\
&\quad +C\Big(\sum\limits_{n=1}^{M}\int_{I_{n}}\|p(y)-P_{h}^{n-1}(y)\|_{0,\Omega^{+}}^{2}dt\Big)^{\frac{1}{2}}\Big(\sum\limits_{n=1}^{M}\int_{I_{n}}\|p(y)-P_{h}^{n-1}(y)\|_{1,\Omega^{+}}^{2}dt\Big)^{\frac{1}{2}}.
\end{align*}
This, together with (\ref{s4666}), (\ref{s4667}) and Theorem 4.1, completes the proof of the theorem.
\end{proof}
\begin{thm}
Let $(\overline{u},\overline{y},\overline{p})$ be the solution of problem \textup{(\ref{s22})-(\ref{s24})} and
$(U_{h},Y_{h},P_{h})$ be the solution of the discretized problem $\textup{(\ref{s33})-(\ref{s36})}$. Then the following estimate holds:
\begin{equation*}
\begin{aligned}
&\sqrt{\alpha} \| \overline{u}-U_{h}\|_{L^{2}(0,T;L^{2}(\Gamma))}+ \|\overline{y}-Y_{h}\|_{L^{2}(0,T;L^{2}(\Omega))}+\|\overline{p}-P_{h}\|_{L^{2}(0,T;L^{2}(\Omega))}\\
&\leq C(\Delta t^{3/4}+h^{3/2}+\Delta t^{1/2}h^{1/2}+\Delta t^{1/4}h).
\end{aligned}
\end{equation*}
\end{thm}
\begin{proof}
It follows from the optimality conditions (\ref{s24}) and
(\ref{s36}) that
\begin{equation*}
\int_{0}^{T}\langle\alpha \overline{u}+\overline{p},U_{h}-\overline{u}\rangle_{\Gamma} dt\geq 0
\end{equation*}
and
\begin{equation*}
\sum\limits_{n=1}^{M}\int_{I_{n}}\langle\alpha U_{h}+P_{h}^{n-1},\overline{u}-U_{h}\rangle_{\Gamma}dt\geq 0.
\end{equation*}
Adding the resulting inequalities leads to
\begin{equation}\label{s410}
\begin{aligned}
\alpha &\sum\limits_{n=1}^{M}\int_{I_{n}}\| \overline{u}-U_{h}\|_{0,\Gamma} ^{2}dt= \sum\limits_{n=1}^{M}\int_{I_{n}} \langle\alpha \overline{u}-\alpha U_{h},\overline{u}-U_{h}\rangle_{\Gamma}dt\\
&=\sum\limits_{n=1}^{M}\int_{I_{n}} \langle\alpha \overline{u}+\overline{p},\overline{u}-U_{h}\rangle_{\Gamma}dt-\sum\limits_{n=1}^{M}\int_{I_{n}} \langle\alpha U_{h}+P_{h}^{n-1},\overline{u}-U_{h}\rangle_{\Gamma}dt\\
&\quad+\sum\limits_{n=1}^{M}\int_{I_{n}} \langle P_{h}^{n-1}-\overline{p},\overline{u}-U_{h}\rangle_{\Gamma}dt\\
&\leq\sum\limits_{n=1}^{M}\int_{I_{n}} \langle P_{h}^{n-1}-\overline{p},\overline{u}-U_{h}\rangle_{\Gamma}dt\\
&\leq\sum\limits_{n=1}^{M}\int_{I_{n}} \langle P_{h}^{n-1}-P_{h}^{n-1}(\overline{y}),\overline{u}-U_{h}\rangle_{\Gamma}dt+\sum\limits_{n=1}^{M}\int_{I_{n}} \langle P_{h}^{n-1}(\overline{y})-\overline{p},\overline{u}-U_{h}\rangle_{\Gamma}dt\\
&=J_{1}+J_{2},
\end{aligned}
\end{equation}
where $P_{h}(\overline{y})$ denotes the solution of (\ref{s42}) with $y=\overline{y}$. From (\ref{s32}), (\ref{s34}), (\ref{s35}) and (\ref{s42}) we have
\begin{equation}\label{s411}
\begin{aligned}
\Big(\frac{Y_{h}^{n}(\overline{u})-Y_{h}^{n}}{\Delta t},w_{h}\Big)-\Big(\frac{Y_{h}^{n-1}(\overline{u})-Y_{h}^{n-1}}{\Delta t},w_{h}\Big)+a(Y_{h}^{n}(\overline{u})-Y_{h}^{n},w_{h})&\\
=\frac{1}{\Delta t}\int_{I_{n}}\langle \overline{u}-U_{h},w_{h}\rangle_{\Gamma}dt&
\end{aligned}
\end{equation}
and
\begin{equation}\label{s412}
\begin{aligned}
\Big(\frac{P_{h}^{n-1}(\overline{y})-P_{h}^{n-1}}{\Delta t},w_{h}\Big)-\Big(\frac{P_{h}^{n}(\overline{y})-P_{h}^{n}}{\Delta t},w_{h}\Big)+a(P_{h}^{n-1}(\overline{y})-P_{h}^{n-1},w_{h})&\\
=\frac{1}{\Delta t}\int_{I_{n}}(\overline{y}-Y_{h}^{n},w_{h})dt&.
\end{aligned}
\end{equation}
Taking $w_{h}=P_{h}^{n-1}-P_{h}^{n-1}(\overline{y})$ in (\ref{s411}) and $w_{h}=Y_{h}^{n}-Y_{h}^{n}(\overline{u})$ in (\ref{s412}), we have
\begin{equation*}
\begin{aligned}
\int_{I_{n}} \langle P_{h}^{n-1}-P_{h}^{n-1}(\overline{y}),\overline{u}-U_{h}\rangle_{\Gamma}dt= \big(Y_{h}^{n-1}-Y_{h}^{n-1}(\overline{u}),P_{h}^{n-1}-P_{h}^{n-1}(\overline{y})\big)&\\
- \big(Y_{h}^{n}-Y_{h}^{n}(\overline{u}),P_{h}^{n}-P_{h}^{n}(\overline{y})\big)+\int_{I_{n}} \big(\overline{y}-Y_{h}^{n},Y_{h}^{n}-Y_{h}^{n}(\overline{u})\big)dt&.
\end{aligned}
\end{equation*}
Summation from $n = 1$ to $M$ leads to
\begin{equation*}
\begin{aligned}
J_{1}&=\sum\limits_{n=1}^{M}\int_{I_{n}} \big(\overline{y}-Y_{h}^{n},Y_{h}^{n}-Y_{h}^{n}(\overline{u})\big)dt\\
&\leq-\frac{1}{2}\sum\limits_{n=1}^{M}\int_{I_{n}} \|\overline{y}-Y_{h}^{n}\|_{0,\Omega}^{2}dt+\frac{1}{2}\sum\limits_{n=1}^{M}\int_{I_{n}}\|\overline{y}-Y_{h}^{n}(\overline{u})\|_{0,\Omega}^{2}dt.
\end{aligned}
\end{equation*}
Using Young's inequality gives
\begin{equation*}
\begin{aligned}
J_{2}&=\sum\limits_{n=1}^{M}\int_{I_{n}} \langle P_{h}^{n-1}(\overline{y})-\overline{p},\overline{u}-U_{h}\rangle_{\Gamma}dt\\
&\leq C\sum\limits_{n=1}^{M}\int_{I_{n}} \|P_{h}^{n-1}(\overline{y})-\overline{p}\|_{0,\Gamma}^{2}dt+\frac{\alpha}{2}\sum\limits_{n=1}^{M}\int_{I_{n}} \|\overline{u}-U_{h}\|^{2}_{0,\Gamma}dt.
\end{aligned}
\end{equation*}
Inserting the estimates for $J_{1}$ and $J_{2}$ into (\ref{s410}) yields
\begin{equation*}
\begin{aligned}
\alpha &\sum\limits_{n=1}^{M}\int_{I_{n}}\| \overline{u}-U_{h}\|_{0,\Gamma} ^{2}dt+\sum\limits_{n=1}^{M}\int_{I_{n}} \|\overline{y}-Y_{h}^{n}\|_{0,\Omega}^{2}dt\\
&\leq C\Big(\sum\limits_{n=1}^{M}\int_{I_{n}}\|\overline{y}-Y_{h}^{n}(\overline{u})\|_{0,\Omega}^{2}dt
+\sum\limits_{n=1}^{M}\int_{I_{n}} \|P_{h}^{n-1}(\overline{y})-\overline{p}\|_{0,\Gamma}^{2}dt\Big).
\end{aligned}
\end{equation*}
Utilizing Theorems 3.1 and 4.1,  we get
\begin{equation}\label{s413}
\begin{aligned}
\alpha &\sum\limits_{n=1}^{M}\int_{I_{n}}\| \overline{u}-U_{h}\|_{0,\Gamma} ^{2}dt+\sum\limits_{n=1}^{M}\int_{I_{n}} \|\overline{y}-Y_{h}^{n}\|_{0,\Omega}^{2}dt\leq C(\Delta t^{3/2}+h^{3}+\Delta th+\Delta t^{1/2}h^{2}).
\end{aligned}
\end{equation}
Setting $w_{h}= P_{h}^{n-1}(\overline{y})-P_{h}^{n-1}$ in (\ref{s412}), we have
\begin{equation*}
\begin{aligned}
&\frac{1}{2}\|P_{h}^{n-1}(\overline{y})-P_{h}^{n-1}\|_{0,\Omega}^{2}-\frac{1}{2}\|P_{h}^{n}(\overline{y})-P_{h}^{n}\|_{0,\Omega}^{2}+\Delta t\|P_{h}^{n-1}(\overline{y})-P_{h}^{n-1}\|_{\varepsilon, \Omega}^{2}\\
&\leq \frac{1}{2}\int_{I_{n}}\|\overline{y}-Y_{h}^{n}\|_{0,\Omega}^{2}dt+\frac{1}{2}\Delta t\|P_{h}^{n-1}(\overline{y})-P_{h}^{n-1}\|_{0,\Omega}^{2}.
\end{aligned}
\end{equation*}
Summing over $n$ from $m $ to $M$, we obtain
\begin{equation*}
\begin{aligned}
&\|P_{h}^{m-1}(\overline{y})-P_{h}^{m-1}\|_{0,\Omega}^{2}+\sum\limits_{n=m}^{M}2\Delta t\|P_{h}^{n-1}(\overline{y})-P_{h}^{n-1}\|_{\varepsilon, \Omega}^{2}\\
&\leq \sum\limits_{n=m}^{M}\int_{I_{n}}\|\overline{y}-Y_{h}^{n}\|_{0,\Omega}^{2}dt+\sum\limits_{n=m}^{M}\Delta t\|P_{h}^{n-1}(\overline{y})-P_{h}^{n-1}\|_{0,\Omega}^{2}.
\end{aligned}
\end{equation*}
Using discrete Gronwall's inequality, we calculate
\begin{equation}\label{s414}
\|P_{h}^{m-1}(\overline{y})-P_{h}^{m-1}\|_{0,\Omega}
\leq C\|\overline{y}-Y_{h}\|_{L^{2}(0,T;L^{2}(\Omega))}.
\end{equation}
Using (\ref{s4666}), (\ref{s4667}), (\ref{s413}), (\ref{s414})  and  Theorem 4.1, we conclude
\begin{equation*}
\begin{aligned}
\sum\limits_{n=1}^{M}\int_{I_{n}} \|\overline{p}-P_{h}^{n-1}\|_{0,\Omega}^{2}dt &\leq\sum\limits_{n=1}^{M}\int_{I_{n}} \|\overline{p}-P_{h}^{n-1}(\overline{y})\|_{0,\Omega}^{2}dt+\sum\limits_{n=1}^{M}\int_{I_{n}} \|P_{h}^{n-1}(\overline{y})-P_{h}^{n-1}\|_{0,\Omega}^{2}dt\\
&\leq \sum\limits_{n=1}^{M}\int_{I_{n}} \|\overline{p}-\overline{R}_{h}^{n}\overline{p}\|_{0,\Omega}^{2}dt+ \sum\limits_{n=1}^{M}\int_{I_{n}} \|\overline{R}_{h}^{n}\overline{p}-P_{h}^{n-1}(\overline{y})\|_{0,\Omega}^{2}dt\\
&\quad+C\|\overline{y}-Y_{h}\|_{L^{2}(0,T;L^{2}(\Omega))}^{2}\\
&\leq C(\Delta t^{3/2}+h^{3}+\Delta th+\Delta t^{1/2}h^{2}).
\end{aligned}
\end{equation*}
Thus, we complete the proof of the theorem.
\end{proof}
\begin{rem}
In this paper, for simplicity, we only consider the case of homogeneous boundary conditions. However, the above theoretical results can be directly applied to nonhomogeneous boundary conditions. We perform numerical experiments for nonhomogeneous boundary conditions to confirm the optimal convergence in Section 5.
\end{rem}
\section{\bf Numerical examples}
In this section we present numerical examples to support our theoretical findings. In all examples, we set the computation domain $\Omega$ as a square $(-1, 1)\times(-1, 1)$ and the regularity parameter $\alpha=1$ and use $N\times N$ uniform triangular meshes and $M$ uniform time grids. We use the fixed-point iteration algorithm to solve the optimal control problem. The algorithm is as follows:\\
\\
\begin{tabular}{cccc}
\toprule
  \leftline{$\mathbf{Algorithm ~1}$ }\\
\midrule
  \leftline{1: Give an initial function  $U_{h,0}\in U_{ad}$;} \\
  \leftline{2: Solve the state equation (\ref{s34}) to obtain $Y_{h,0}^{n},~ n=1,\cdots,M$;}\\
  \leftline{3: Solve the adjoint state equation (\ref{s35}) to obtain $P_{h,0}^{n-1},~ n=1,\cdots,M$. Set $k=0$;}\\
   \leftline{4: $\mathbf{repeat}$;}\\
  \leftline{5: Set $U_{h,k}^{n}=P_{U_{ad}}\big(-\frac{1}{\alpha}P_{h,k-1}^{n-1}|_{\Gamma}\big)$ ;}\\
  \leftline{6: Solve the state equation (\ref{s34}) to obtain $Y_{h,k}^{n},~ n=1,\cdots,M$;}\\
  \leftline{7: Solve the adjoint state equation (\ref{s35}) to obtain $P_{h,k}^{n-1},~ n=1,\cdots,M$. Set $k=k+1$;}\\
  \leftline{8: $\mathbf{until}$ stopping criteria.}\\
\bottomrule
\end{tabular}
\\
\\
This algorithm is convergent if the regularity parameter $\alpha$ is large enough (see, e.g., \cite{Hinze2009}).
In the following numerical examples, we define the experimental order of convergence by
$$\text{Order}=\frac{\log E(h_{1}) -\log E(h_{2})}{\log h_{1}-\log h_{2}},$$
where $E(h)$ denotes the error on triangulation with mesh size $h$.
\begin{rem}
Although we assumed in the previous sections that $\Omega$ is a bounded domain with smooth boundary, numerical experiments demonstrate that the algorithm remains effective in non-smooth domains.
\end{rem}
\noindent\textbf{Example 1.}
In this example, we consider a circle interface
$\Gamma=\{(x_1,x_2): x_{1}^2+x_{2}^{2}-r_{0}^2=0\}$ with $r_{0}=0.5$ (see Fig. 3).
\begin{center}
\includegraphics[width=2in,height=2in]{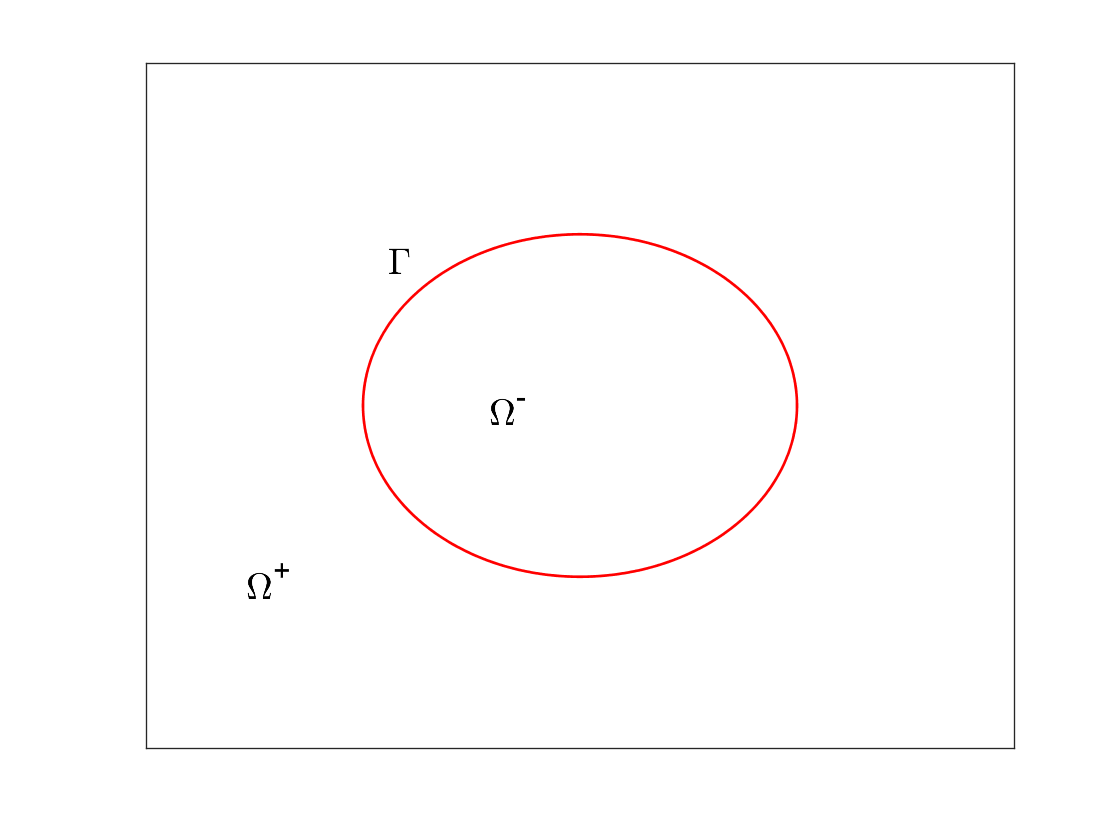}\\
{\mbox{\footnotesize Fig. 3. The circle interface for Example 1.}}
\end{center}
We choose $u_{a}= t(\sin(\pi x_{1})-\cos(\pi x_{2}))$ and $u_{b}=t(x_{1}^{2}+x_{2})$. The exact solution $\overline{y}$ is constructed with a nonhomogeneous boundary condition.
The optimal triple $(\overline{y},\overline{p},\overline{u})$ is given by
\begin{equation*}
\begin{aligned}
&\overline{y}(x_1,x_2,t)=\left\{\begin{array}{ll}
e^{t}\Big(\frac{\big(x_1^2+x_2^2\big)^{3/2}}{\beta^{-}}+\frac{1}{4\beta^{-}}\big(\frac{x_1^2+x_2^2}{r_{0}^{2}}-1\big)\Big), & \text { in  }~\Omega^-, \\
e^{t}\Big(\frac{\big(x_1^2+x_2^2\big)^{3/2}}{\beta^{+}}+\big(\frac{1}{\beta^{-}}-\frac{1}{\beta^{+}}\big)r_{0}^3\Big),  & \text { in  }~\Omega^+;
\end{array}\right.\\
&\overline{p}(x_1,x_2,t)=\left\{\begin{array}{ll}
(t-1)(x_1^2+x_2^2-r_{0}^2)(x_{1}^2-1)(x_{2}^2-1)/\beta^-, & \text { in }~\Omega^-, \\
(t-1)(x_1^2+x_2^2-r_{0}^2)(x_{1}^2-1)(x_{2}^2-1)/\beta^+, & \text { in  }~ \Omega^+;
\end{array}\right.\\
&\overline{u}(x_1,x_2,t)=\max\big\{t(\sin(\pi x_1)-\cos(\pi x_2)),\min\{t(x_1^2+x_2), 0\}\big\}.
\end{aligned}
\end{equation*}
We test the convergence performance for both small and large jumps, namely $\beta^-/\beta^+ = 1/10, 10/ 1, 1/1000, 1000/1$. The time step is chosen
as $\Delta t=O(h^2)$, where $\Delta t$ is the time step size and $h$ is the space mesh size. The errors and their convergence orders are shown in Table $1$-Table $4$.
We see from Table $1$ to Table $4$ that the convergence order for the state,
control, and adjoint state is second, which is better than our theoretical result. The exact solution and the computed solution
images of the state, adjoint state, and control with $N= 128$ and $ M=4096$ are shown in Figs. 4-6. From these figures it is observed that the approximate solution is almost identical to the exact solution. The error images of the state, adjoint state, and control with $N= 128$ and $ M=4096$ are shown in Fig. 7. We can find that the numerical errors are mainly accumulated on the interface, which is consistent with our prediction.

\begin{table}[H]
		\centering		
		\caption{The $L^{2}$ error and convergence order of the state, control, and adjoint state for Example 1 with $\beta^-=1$ and $\beta^+=10$.}
		\label{table11}
		\resizebox{1\textwidth}{!}{
		\begin{tabular}{*{7}{c}}
\bottomrule
\multirow{2}*{$1/h$} & \multicolumn{2}{c}{state}&
\multicolumn{2}{c}{control}&
\multicolumn{2}{c}{adjoint state}\\
\cmidrule(lr){2-3}\cmidrule(lr){4-5}\cmidrule(lr){6-7}
    &$\|\overline{y}-Y_{h}\|_{L^{2}(0,T;L^{2}(\Omega))}$& Order & $\|\overline{u}-U_{h}\|_{L^{2}(0,T;L^{2}(\Gamma))}$ & Order & $\|\overline{p}-P_{h}\|_{L^{2}(0,T;L^{2}(\Omega))}$ & Order \\
  \midrule
8 & $      2.0664    \text{E-02}	  $ & $\setminus$ & $  6.8295\text{E-04} $ & $  \setminus$& $     	3.2776 \text{E-03}    $& $  \setminus     $ \\
  16 & $   4.9825 \text{E-03}	 $ & $  2.0522$ & $     2.1094\text{E-04}$ & $   1.6950
  $& $    7.9271\text{E-04}	 $& $      2.0478    $ \\
  32 & $  1.2533 \text{E-03}	 $ & $ 1.9911$  & $  5.4057\text{E-05}$ & $   1.9643
  $& $     2.0154\text{E-04}	   $& $     1.9758       $ \\
  64 & $  2.9749  \text{E-04}	 $ & $ 2.0748$  & $   1.3587\text{E-05}$ & $   1.9923
  $& $     4.8446\text{E-05}	   $& $     2.0566       $ \\
  \bottomrule
\end{tabular}
		}	
\end{table}
\begin{table}[H]
		\centering		
		\caption{ The $L^{2}$ error and convergence order of the state, control, and adjoint state for Example 1 with $\beta^-=10$ and $\beta^+=1$.}
		\label{table11}
		\resizebox{1\textwidth}{!}{
		\begin{tabular}{*{7}{c}}
\bottomrule
\multirow{2}*{$1/h$} & \multicolumn{2}{c}{state}&
\multicolumn{2}{c}{control}&
\multicolumn{2}{c}{adjoint state}\\
\cmidrule(lr){2-3}\cmidrule(lr){4-5}\cmidrule(lr){6-7}
    &$\|\overline{y}-Y_{h}\|_{L^{2}(0,T;L^{2}(\Omega))}$& Order & $\|\overline{u}-U_{h}\|_{L^{2}(0,T;L^{2}(\Gamma))}$ & Order & $\|\overline{p}-P_{h}\|_{L^{2}(0,T;L^{2}(\Omega))}$ & Order \\
  \midrule
8 & $     4.5819  \text{E-02}	  $ & $  \setminus $ & $   3.8888\text{E-03} $ & $ \setminus $& $     	1.1681\text{E-02}    $& $   \setminus $ \\
  16 & $   1.1486  \text{E-02}	 $ & $    1.9961 $ & $  9.9305\text{E-04}$ & $ 1.9694
    $& $   2.9565\text{E-03}	 $& $     1.9821    $ \\
  32 & $  2.8399  \text{E-03}	 $ & $   2.0160 $  & $2.4745\text{E-04}$ & $ 2.0047 $& $    7.3640 \text{E-04}	   $& $      2.0053    $ \\
  64 & $ 7.1584\text{E-04}$ & $  1.9881 $  & $6.3766\text{E-05}$ & $1.9563  $& $ 1.8544\text{E-04} $& $ 1.9895 $ \\
  \bottomrule
\end{tabular}
		}	
\end{table}
\begin{table}[H]
		\centering		
		\caption{ The $L^{2}$ error and convergence order of the state, control, and adjoint state for Example 1 with $\beta^-=1$ and $\beta^+=1000$.}
		\label{table11}
		\resizebox{1\textwidth}{!}{
		\begin{tabular}{*{7}{c}}
\bottomrule
\multirow{2}*{$1/h$} & \multicolumn{2}{c}{state}&
\multicolumn{2}{c}{control}&
\multicolumn{2}{c}{adjoint state}\\
\cmidrule(lr){2-3}\cmidrule(lr){4-5}\cmidrule(lr){6-7}
    &$\|\overline{y}-Y_{h}\|_{L^{2}(0,T;L^{2}(\Omega))}$& Order & $\|\overline{u}-U_{h}\|_{L^{2}(0,T;L^{2}(\Gamma))}$ & Order & $\|\overline{p}-P_{h}\|_{L^{2}(0,T;L^{2}(\Omega))}$ & Order \\
  \midrule
8 & $     1.3263 \text{E-01}	  $ & $ \setminus $ & $  4.2117\text{E-04} $ & $\setminus $& $     	 2.0023  \text{E-02}    $& $  \setminus    $ \\
  16 & $   3.9555 \text{E-02}	 $ & $  1.7454  $ & $   1.8196  \text{E-04}$ & $   1.2108
    $& $   5.3393 \text{E-03}	 $& $     1.9069    $ \\
  32 & $  8.6116\text{E-03}	 $ & $  2.1995$  & $ 5.9664\text{E-05}$ & $   1.6087$& $    1.1303 \text{E-03}	   $& $   2.2399     $ \\
  64 & $   1.5449\text{E-03}	 $ & $   2.4787  $  & $ 1.6461\text{E-05}$ & $    1.8579$& $   2.0294\text{E-04}	   $& $     2.4776    $ \\
  \bottomrule
\end{tabular}
		}	
\end{table}
\begin{table}[H]
		\centering		
		\caption{ The $L^{2}$ error and convergence order of the state, control, and adjoint state for Example 1 with $\beta^-=1000$ and $\beta^+=1$.}
		\label{table11}
		\resizebox{1\textwidth}{!}{
		\begin{tabular}{*{7}{c}}
\bottomrule
\multirow{2}*{$1/h$} & \multicolumn{2}{c}{state}&
\multicolumn{2}{c}{control}&
\multicolumn{2}{c}{adjoint state}\\
\cmidrule(lr){2-3}\cmidrule(lr){4-5}\cmidrule(lr){6-7}
    &$\|\overline{y}-Y_{h}\|_{L^{2}(0,T;L^{2}(\Omega))}$& Order & $\|\overline{u}-U_{h}\|_{L^{2}(0,T;L^{2}(\Gamma))}$ & Order & $\|\overline{p}-P_{h}\|_{L^{2}(0,T;L^{2}(\Omega))}$ & Order \\
 \midrule
8 & $    4.5610  \text{E-02}	  $ & $ \setminus$ & $  3.6907\text{E-03} $ & $\setminus $& $     	1.1554\text{E-02} $& $   \setminus   $ \\
  16 & $    1.1452   \text{E-02}	 $ & $   1.9938$ & $   9.4677\text{E-04}$ & $1.9628$& $   2.9353\text{E-03}	 $& $      1.9768  $ \\
  32 & $  2.8317\text{E-03}	 $ & $ 2.0159 $  & $2.3853\text{E-04}$ & $  1.9888$& $    7.3243\text{E-04}	   $& $    2.0027$ \\
  64 & $  7.1498   \text{E-04}	 $ & $ 1.9857$  & $6.0223\text{E-05}$ & $  1.9858$& $    1.8445\text{E-04}	   $& $    1.9895 $ \\
  \bottomrule
\end{tabular}
		}	
\end{table}
\begin{center}
\includegraphics[width=5in,height=2.5in]{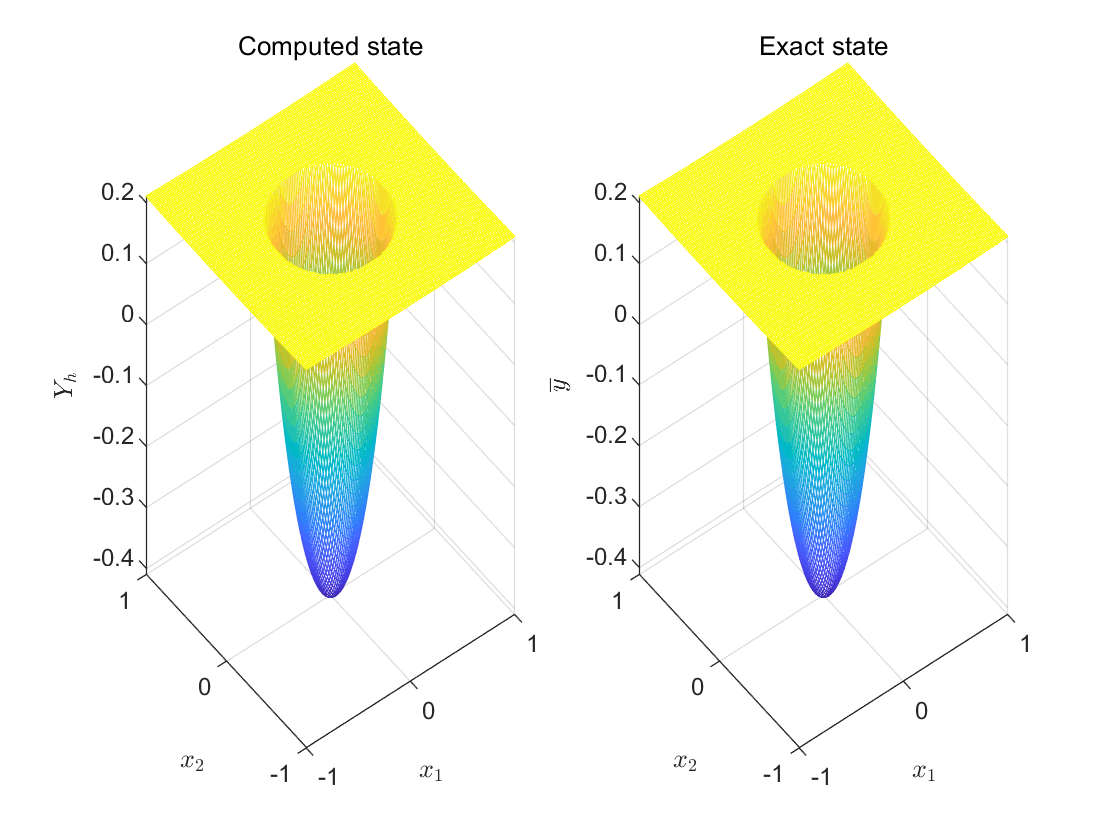}\\
{\mbox{\footnotesize Fig. 4. The computed state and the exact state with  $\beta^{-}/\beta^{+} = 1/1000$.}}
\end{center}
\begin{center}
\includegraphics[width=5in,height=2.5in]{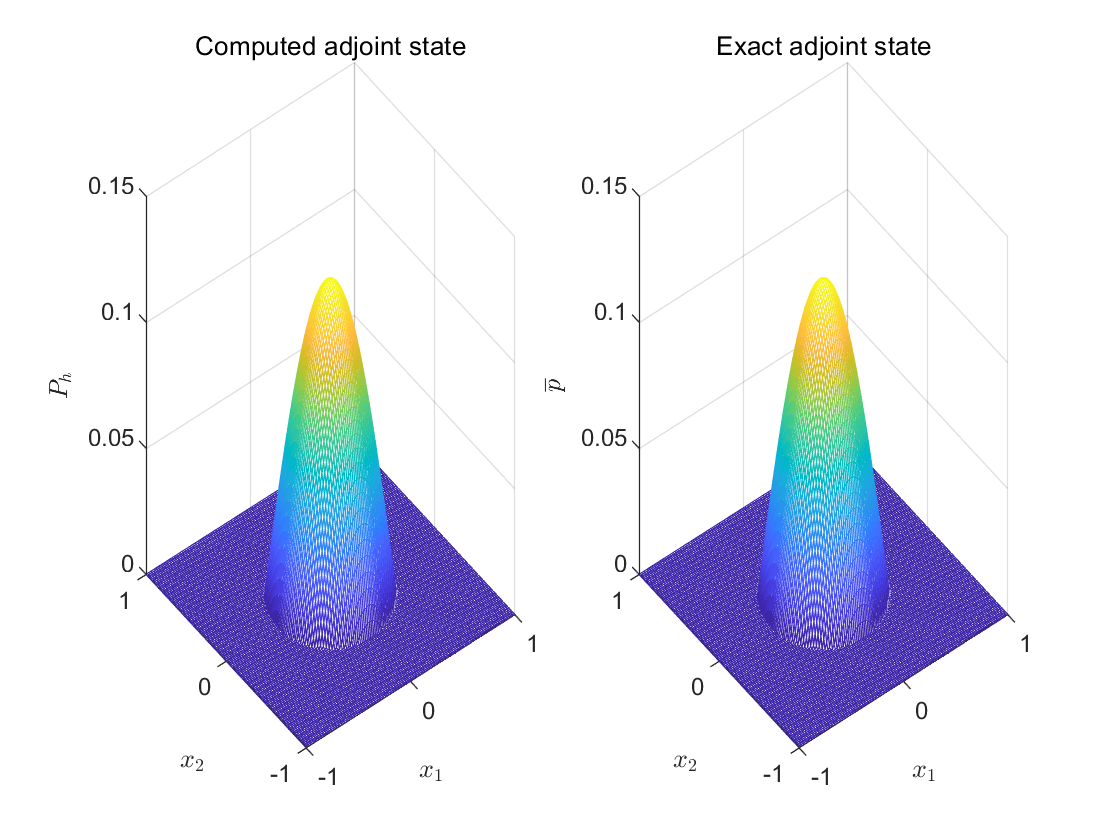}\\
{\mbox{\footnotesize Fig. 5. The computed adjoint state and the exact adjoint state with $\beta^{-}/\beta^{+} = 1/1000$.}}
\end{center}
\begin{center}
\includegraphics[width=5in,height=2.5in]{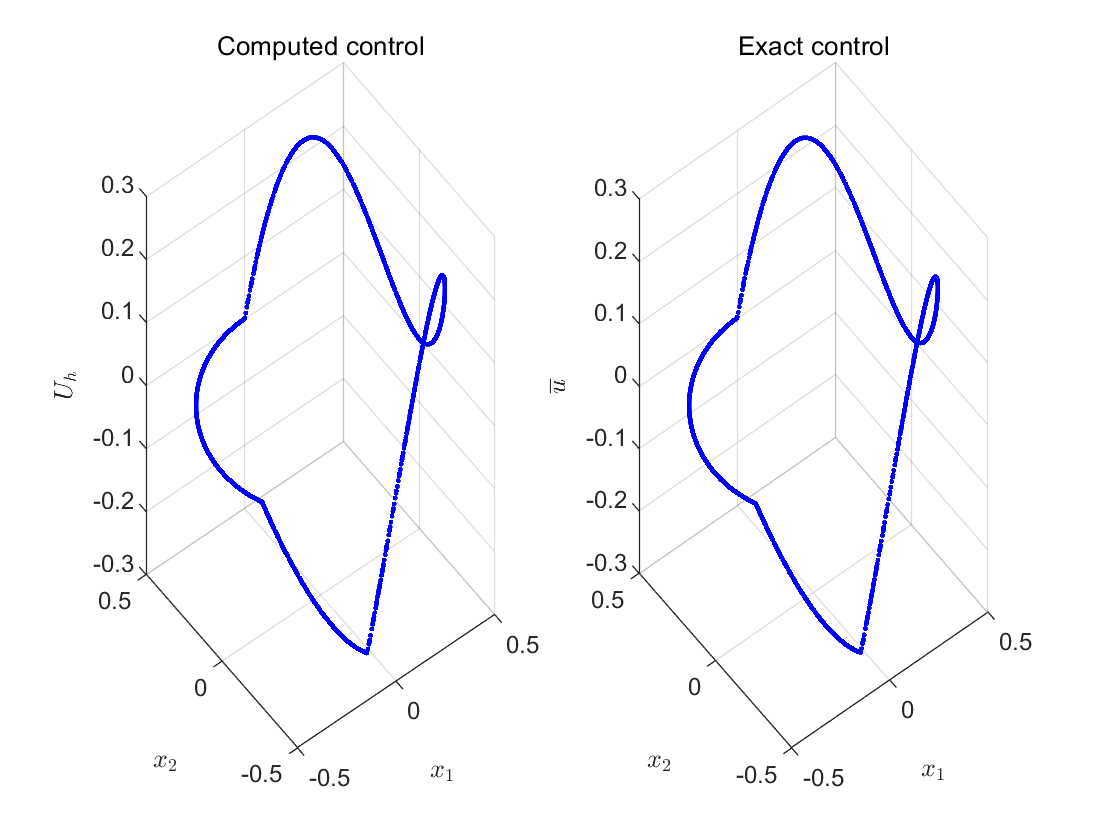}\\
{\mbox{\footnotesize Fig. 6. The computed control and the exact control with $\beta^{-}/\beta^{+} = 1/1000$.}}
\end{center}
\begin{figure}[H]
	\centering
	\begin{subfigure}{0.325\linewidth}
		\centering
		\includegraphics [width=2in,height=2in]{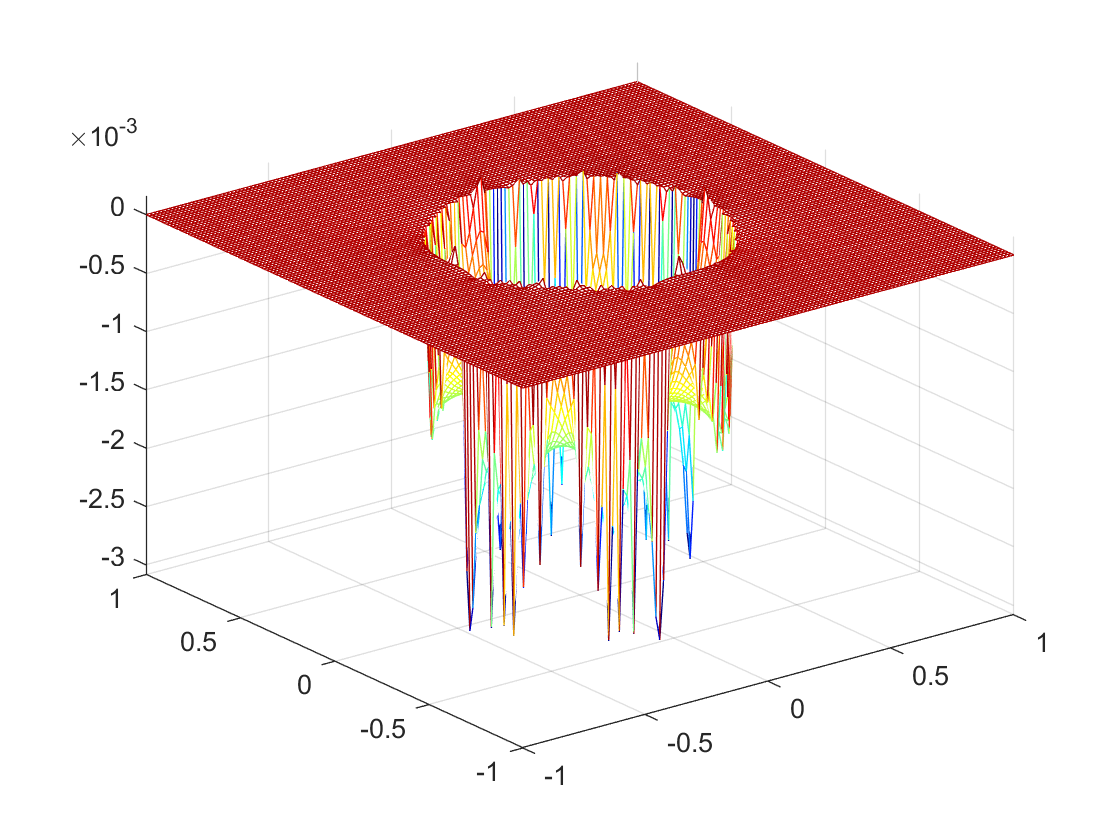}
		\subcaption*{\footnotesize (a) state error}
		\label{chutian3}
	\end{subfigure}
	\centering
	\begin{subfigure}{0.325\linewidth}
		\centering
		\includegraphics [width=2in,height=2in]{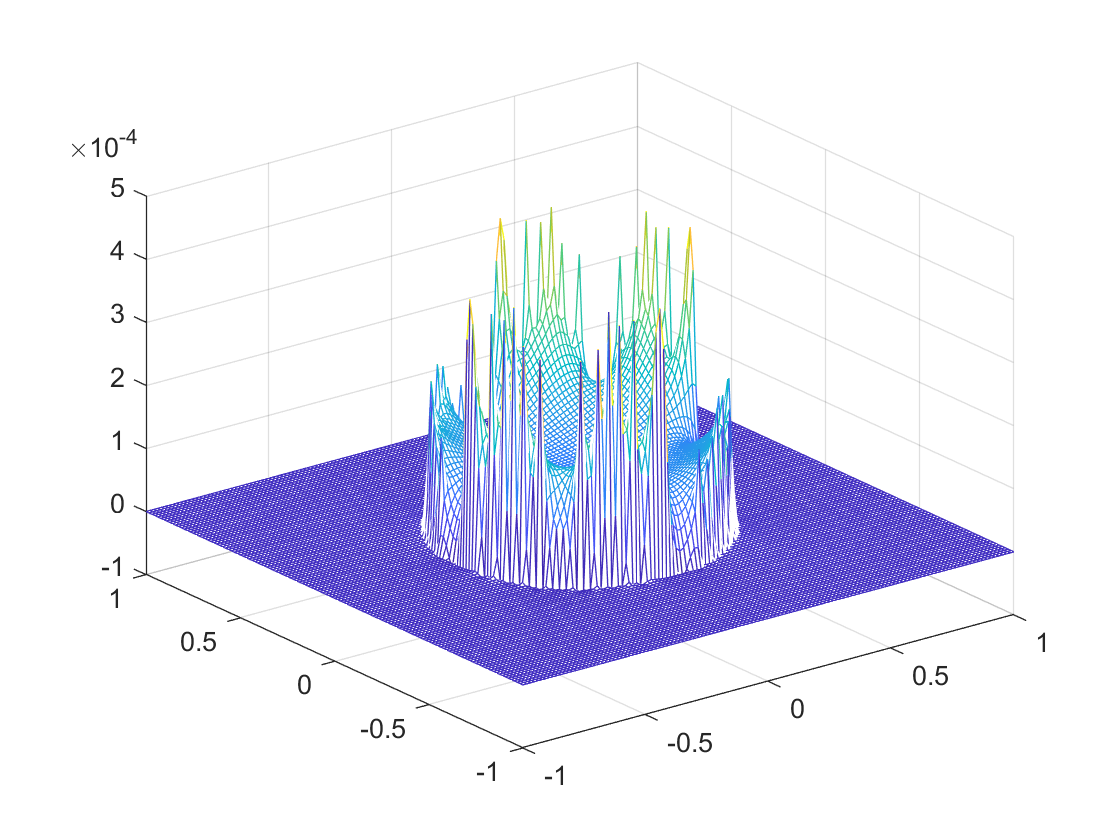}
		\subcaption*{\footnotesize (b) adjoint state error}
		\label{chutian3}
	\end{subfigure}
	\centering
	\begin{subfigure}{0.325\linewidth}
		\centering
		\includegraphics [width=2in,height=2in]{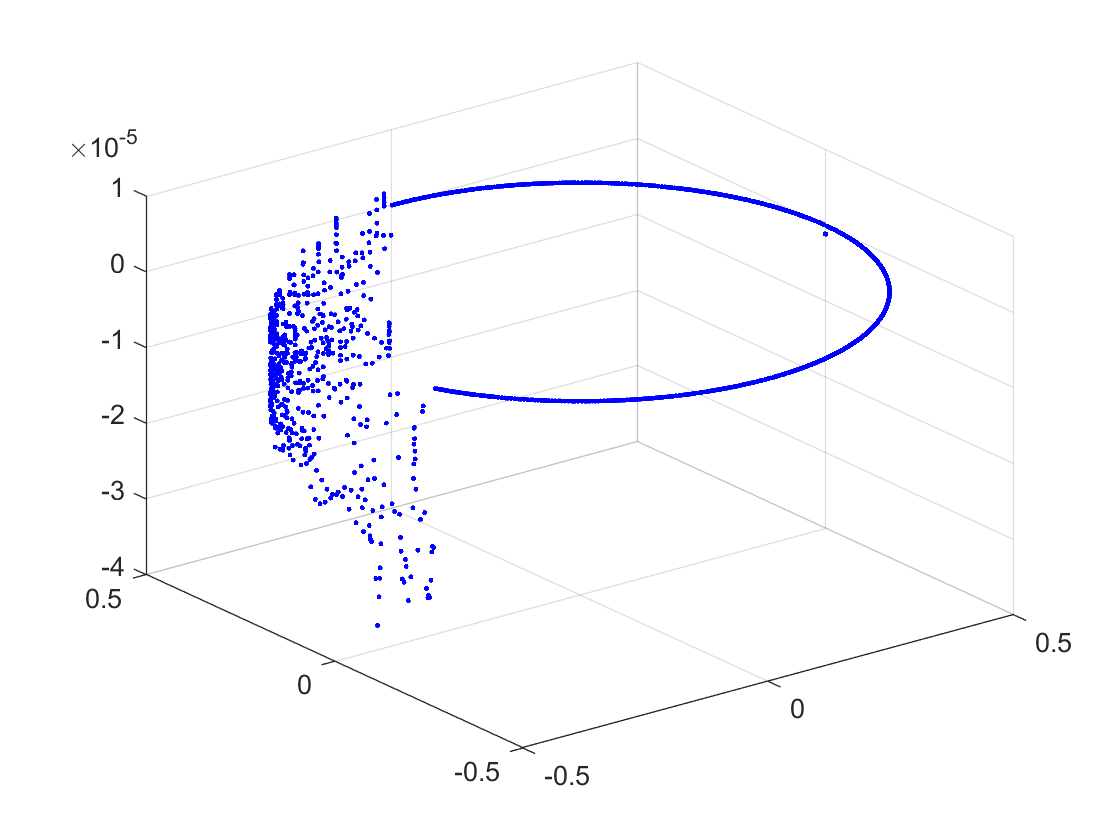}
		\subcaption*{\footnotesize (c) control error}
		\label{chutian3}
	\end{subfigure}
	\caption*{\footnotesize Fig. 7. The error of the state, adjoint state, and control with $\beta^{-}/\beta^{+} = 1/1000$ for Example 1.}
	\label{da_chutian}
\end{figure}
\noindent\textbf{Example 2.} In this example, we consider a cubic curve \cite{Guo2018},  i.e., $\Gamma =\{(x_1,x_2): x_2-3x_1(x_1-0.3)(x_1-0.8)-0.38=0\}$ (see Fig. 8). We consider both constrained and unconstrained cases.
\begin{center}
\includegraphics[width=2in,height=2in]{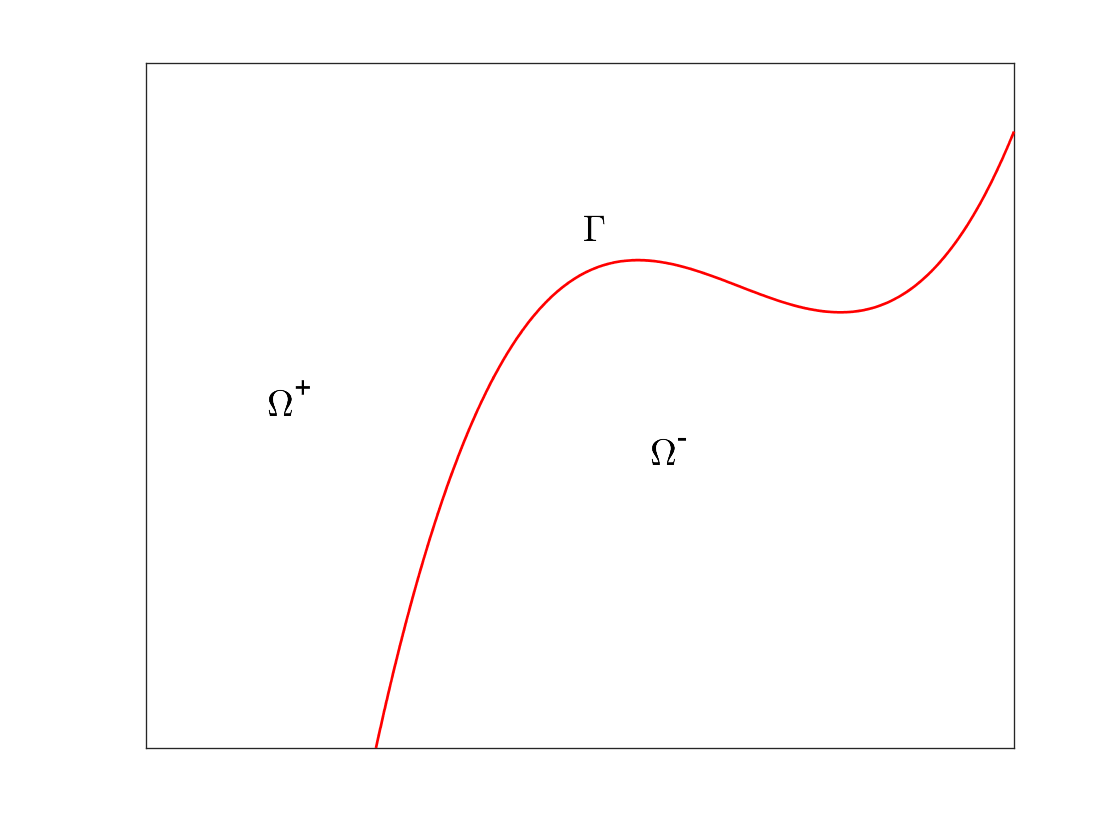}\\
{\mbox{\footnotesize Fig. 8. The cubic curve interface for Example 2.}}
\end{center}
\noindent\textbf{Case 1} In this case, we consider problems without control constraints. The optimal triple $(\overline{y},\overline{p},\overline{u})$ is given by
\begin{equation*}
\begin{aligned}
&\overline{y}(x_1,x_2,t)=\left\{\begin{array}{ll}
\cos(t-1)\big(-3x_{1}^3+x_{2}^2-0.38\big), & \text { in  }~\Omega^-, \\
\cos(t-1)\big(-x_{2}+x_{2}^2-3.3x_{1}^2+0.72x_{1}\big),  & \text { in  }~ \Omega^+;
\end{array}\right.\\
&\overline{p}(x_1,x_2,t)=\left\{\begin{array}{ll}
\sin(t-1)\big(x_{2}-3x_{1}^3+3.3x_{1}^2-0.72x_{1}-0.38\big)(x_{1}^2-1)(x_{2}^2-1)/\beta^-, &\text { in }~\Omega^-, \\
\sin(t-1)\big(x_{2}-3x_{1}^3+3.3x_{1}^2-0.72x_{1}-0.38\big)(x_{1}^2-1)(x_{2}^2-1)/\beta^+, &\text { in  }~\Omega^+;
\end{array}\right.\\
&\overline{u}(x_1,x_2,t)=0.
\end{aligned}
\end{equation*}
\noindent\textbf{Case 2} In this case, we consider problems with control constraints.
 We set $u_{a}=t(x_{2}-3x_{1}^3+0.3x_{1}^2)$ and $u_{b}=1$. The control variable $\overline{u}$ as follows:
$$\overline{u}(x_1,x_2,t)=\max\{t(x_{2}-3x_{1}^3+0.3x_{1}^2),0\}.$$
Other data are set as in Case 1.

Consider the cases in which the discontinuous diffusion coefficient is $\beta^{-}/\beta^{+}=1/10$. At first, we set $\Delta t=O(h^2)$. The $L^{2}$ norm error and convergence order of the control, state, and adjoint state without and with control constraint are shown in Tables 5-6. From Tables 5-6, we observe that the convergence order is second for the control, state, and adjoint state. Then we set $\Delta t=O(h)$ and present the errors of the control, state, and adjoint state in Tables 7 and 8. We find that the convergence order is first for the control, state, and adjoint state. The exact solution and the computed solution images of the state, adjoint state, and control with $N= 128$ and $ M=4096$ are shown in Figs. 9-11 and Figs. 13-15 for the unconstrained and constrained cases, respectively. The error images of the state, adjoint state, and control with $N= 128$ and $ M=4096$ are shown in Fig. 12 and Fig. 16 for the unconstrained and constrained cases, respectively. From these results, the numerical approach seems to be applicable to the case of $ \Omega\cap \Gamma\neq 0$.
\begin{table}[H]
		\centering		
		\caption{ The $L^{2}$ error and convergence order of the state, control, and adjoint state for Example 2 with $\beta^-=1$ and $\beta^+=10$ (without control constraints).}
		\label{table11}
		\resizebox{1\textwidth}{!}{
		\begin{tabular}{*{7}{c}}
\bottomrule
\multirow{2}*{$1/h$} & \multicolumn{2}{c}{state}&
\multicolumn{2}{c}{control}&
\multicolumn{2}{c}{adjoint state}\\
\cmidrule(lr){2-3}\cmidrule(lr){4-5}\cmidrule(lr){6-7}
    &$\|\overline{y}-Y_{h}\|_{L^{2}(0,T;L^{2}(\Omega))}$& Order & $\|\overline{u}-U_{h}\|_{L^{2}(0,T;L^{2}(\Gamma))}$ & Order & $\|\overline{p}-P_{h}\|_{L^{2}(0,T;L^{2}(\Omega))}$ & Order \\
  \midrule
 8 & $     2.4971 \text{E-02}	  $ & $  \setminus $ & $  2.3242\text{E-03} $ & $\setminus   $& $     1.2054\text{E-02}    $& $  \setminus    $ \\
  16 & $   6.8800 \text{E-03}	 $ & $   1.8598 $ & $   5.7087 \text{E-04}$ & $ 2.0255   $& $   2.9320  \text{E-03}	 $& $      2.0396   $ \\
  32 & $  1.7474 \text{E-03}	 $ & $  1.9772$  & $1.5415\text{E-04}$ & $  1.8889   $& $    7.2602\text{E-04}	   $& $     2.0138        $ \\
  64 & $   4.7973  \text{E-04}	 $ & $ 1.8660$  & $3.9753\text{E-05}$ & $  1.9552    $& $     1.7898 \text{E-04}	   $& $      2.0202      $ \\
  \bottomrule
\end{tabular}
		}	
\end{table}
\begin{table}[H]
		\centering		
		\caption{ The $L^{2}$ error and convergence order of the state, control, and adjoint state for Example 2 with $\beta^-=1$ and $\beta^+=10$ (with control constraints).}
		\label{table11}
		\resizebox{1\textwidth}{!}{
		\begin{tabular}{*{7}{c}}
\bottomrule
\multirow{2}*{$1/h$} & \multicolumn{2}{c}{state}&
\multicolumn{2}{c}{control}&
\multicolumn{2}{c}{adjoint state}\\
\cmidrule(lr){2-3}\cmidrule(lr){4-5}\cmidrule(lr){6-7}
    &$\|\overline{y}-Y_{h}\|_{L^{2}(0,T;L^{2}(\Omega))}$& Order & $\|\overline{u}-U_{h}\|_{L^{2}(0,T;L^{2}(\Gamma))}$ & Order & $\|\overline{p}-P_{h}\|_{L^{2}(0,T;L^{2}(\Omega))}$ & Order \\
   \midrule
8 & $     2.4977   \text{E-02}	  $ & $  \setminus$ & $ 1.9253\text{E-03} $ & $ \setminus$& $     	1.2055 \text{E-02}    $& $  \setminus     $ \\
  16 & $  6.8795\text{E-03}	 $ & $   1.8602  $ & $   4.5507\text{E-04}$ & $  2.0810 $& $   2.9319\text{E-03}	 $& $       2.0397    $ \\
  32 & $  1.7469 \text{E-03}	 $ & $ 1.9775$  & $  1.1987\text{E-04}$ & $ 1.9247$& $    7.2600\text{E-04}	   $& $      2.0138      $ \\
  64 & $  4.7957\text{E-04}	 $ & $ 1.8643$  & $ 3.0253\text{E-05}$ & $ 1.9863$& $    1.7897 \text{E-04}	   $& $      2.0203      $ \\
  \bottomrule
\end{tabular}
		}	
\end{table}
\begin{table}[H]
		\centering		
		\caption{ The $L^{2}$ error and convergence order of the state, control, and adjoint state for Example 2 with $\beta^-=1$ and $\beta^+=10$ (without control constraints).}
		\label{table11}
		\resizebox{1\textwidth}{!}{
		\begin{tabular}{*{7}{c}}
\bottomrule
\multirow{2}*{$1/h$} & \multicolumn{2}{c}{state}&
\multicolumn{2}{c}{control}&
\multicolumn{2}{c}{adjoint state}\\
\cmidrule(lr){2-3}\cmidrule(lr){4-5}\cmidrule(lr){6-7}
    &$\|\overline{y}-Y_{h}\|_{L^{2}(0,T;L^{2}(\Omega))}$& Order & $\|\overline{u}-U_{h}\|_{L^{2}(0,T;L^{2}(\Gamma))}$ & Order & $\|\overline{p}-P_{h}\|_{L^{2}(0,T;L^{2}(\Omega))}$ & Order \\
  \midrule
 8 & $  7.1119\text{E-02}	  $ & $ \setminus$ & $2.9211\text{E-03}$ & $ \setminus$& $     	 3.4054 \text{E-02}    $& $   \setminus  $ \\
  16 & $3.3425 \text{E-02}	 $ & $   1.0893  $ & $ 1.5314\text{E-03}	$ & $ 0.9317 $& $   1.5249 \text{E-02}	 $& $   1.1592   $ \\
  32 & $   1.6087\text{E-02}	 $ & $ 1.0551$  & $ 8.0592\text{E-04}	 $ & $ 0.9262$& $       7.2486\text{E-03}	   $& $     1.0729  $ \\
  64 & $   7.9432\text{E-03}	 $ & $ 1.0181$  & $4.1058\text{E-04}	 $ & $ 0.9730$& $       3.5382\text{E-03}	   $& $     1.0347 $ \\
  \bottomrule
\end{tabular}
		}	
\end{table}
\begin{table}[H]
		\centering		
		\caption{ The $L^{2}$ error and convergence order of the state, control, and adjoint state for Example 2 with $\beta^-=1$ and $\beta^+=10$ (with control constraints).}
		\label{table11}
		\resizebox{1\textwidth}{!}{
		\begin{tabular}{*{7}{c}}
\bottomrule
\multirow{2}*{$1/h$} & \multicolumn{2}{c}{state}&
\multicolumn{2}{c}{control}&
\multicolumn{2}{c}{adjoint state}\\
\cmidrule(lr){2-3}\cmidrule(lr){4-5}\cmidrule(lr){6-7}
    &$\|\overline{y}-Y_{h}\|_{L^{2}(0,T;L^{2}(\Omega))}$& Order & $\|\overline{u}-U_{h}\|_{L^{2}(0,T;L^{2}(\Gamma))}$ & Order & $\|\overline{p}-P_{h}\|_{L^{2}(0,T;L^{2}(\Omega))}$ & Order \\
  \midrule
  8 & $     7.1096 \text{E-02}	  $ & $   \setminus$ & $  2.5960\text{E-03} $ & $ \setminus $& $     	3.4053 \text{E-02}    $& $  \setminus  $ \\
  16 & $   3.3405\text{E-02}	 $ & $  1.0897  $ & $  1.2000\text{E-03}$ & $ 1.1133
    $& $    1.5248\text{E-02}	 $& $     1.1592   $ \\
  32 & $  1.6075  \text{E-02}	 $ & $ 1.0552  $  & $  5.9522\text{E-04}$ & $ 1.0115
    $& $  7.2481\text{E-03}	   $& $      1.0729        $ \\
  64 & $  7.9372   \text{E-03}	 $ & $  1.0181   $  & $ 2.9690\text{E-04}$ & $  1.0035
    $& $   3.5379\text{E-03}	   $& $   1.0347      $ \\
  \bottomrule
\end{tabular}
		}	
\end{table}
\begin{center}
\includegraphics[width=5in,height=2.5in]{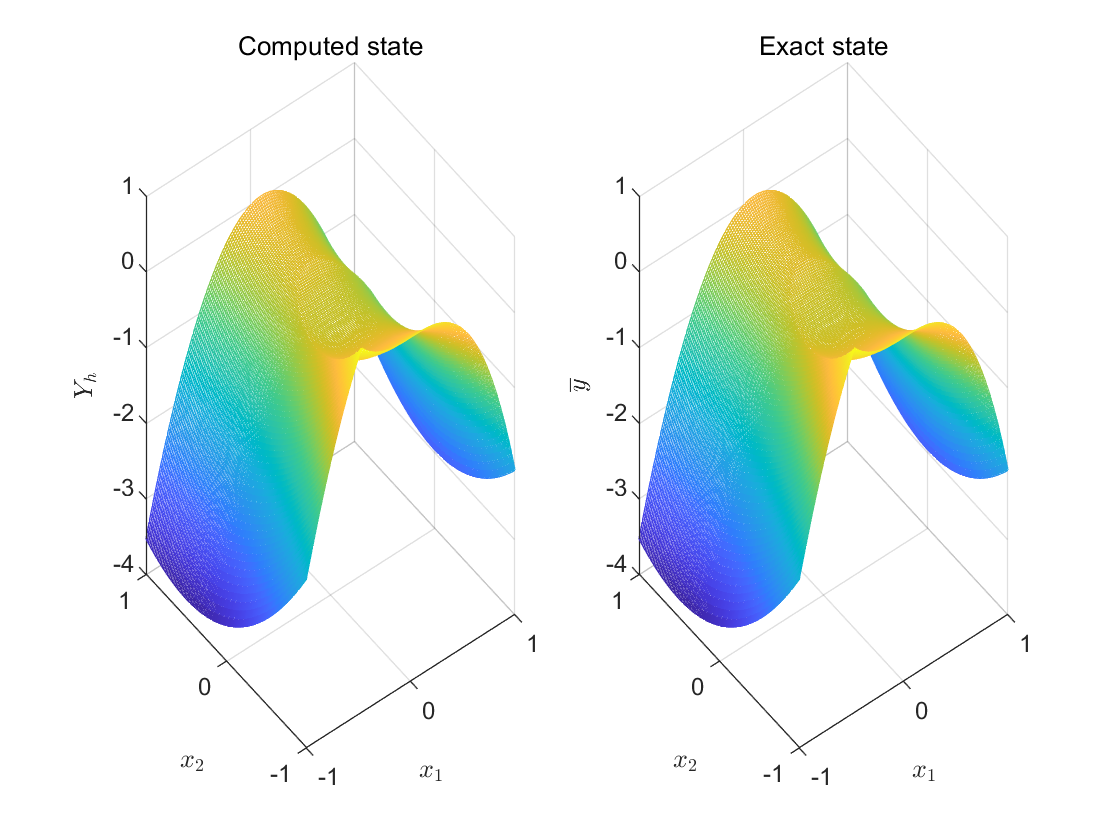}\\
{\mbox{\footnotesize Fig. 9. The computed state and the exact state (without control constraints).}}
\end{center}
\begin{center}
\includegraphics[width=5in,height=2.5in]{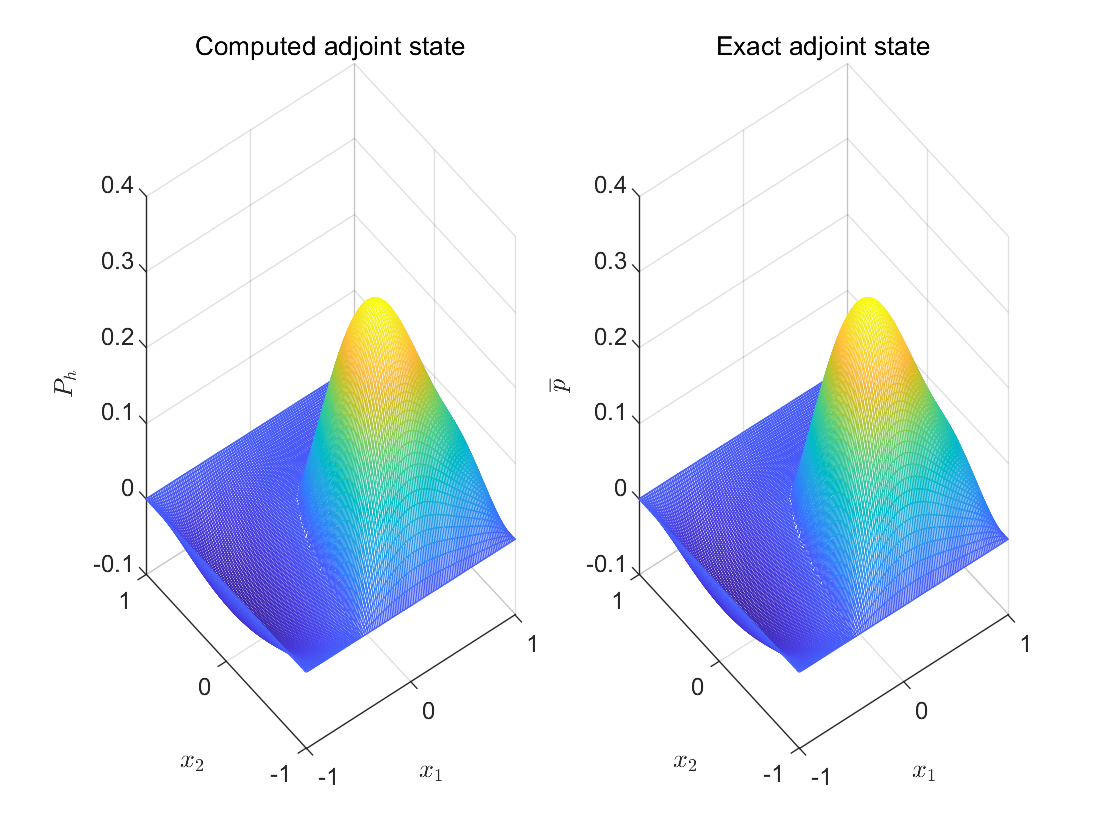}\\
{\mbox{\footnotesize Fig. 10. The computed adjoint state and the exact adjoint state (without control constraints).}}
\end{center}
\begin{center}
\includegraphics[width=5in,height=2.5in]{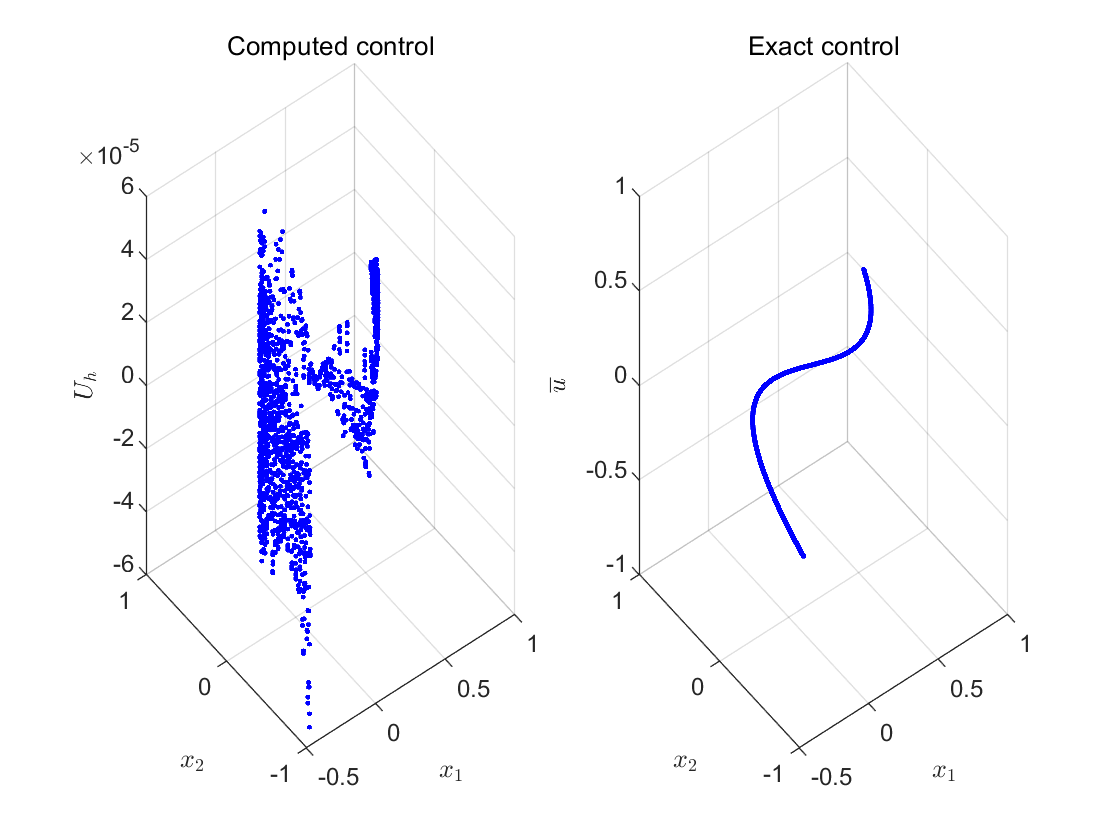}\\
{\mbox{\footnotesize Fig. 11. The computed control and the exact control (without control constraints).}}
\end{center}
\begin{figure}[H]
	\centering
	\begin{subfigure}{0.325\linewidth}
		\centering
		\includegraphics [width=2in,height=2in]{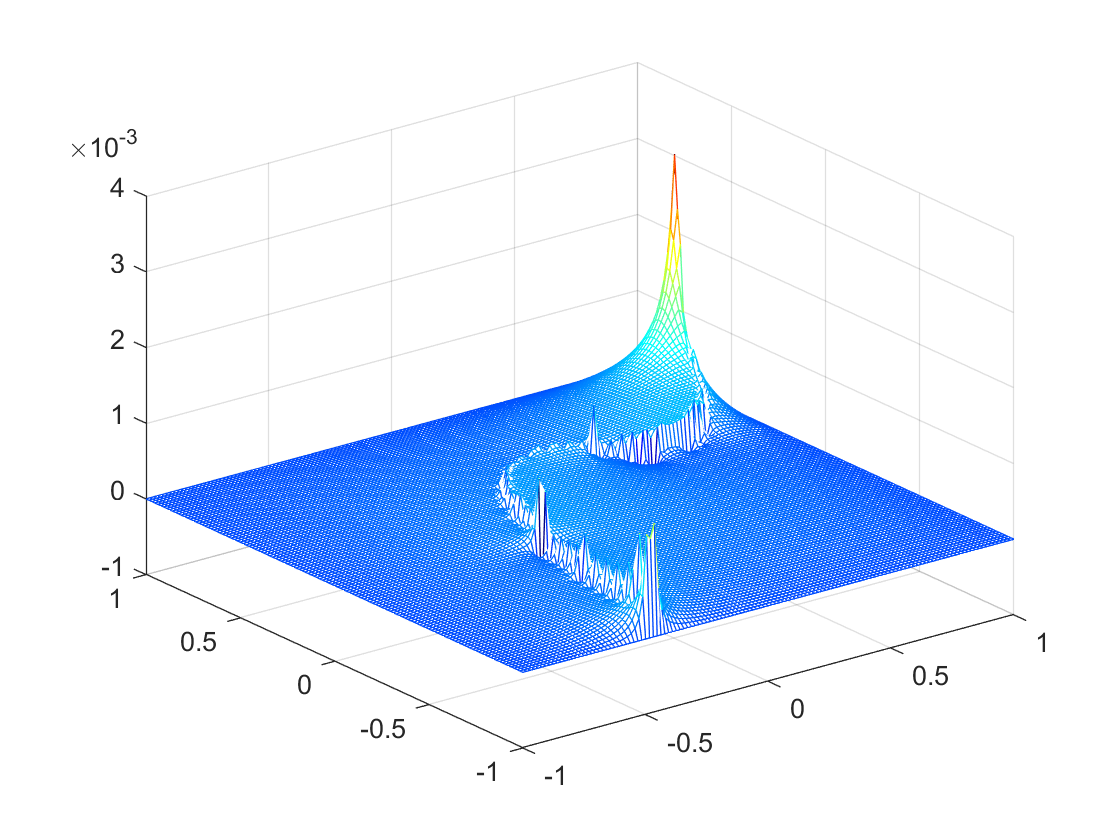}
		\subcaption*{\footnotesize (a) state error}
		\label{chutian3}
	\end{subfigure}
	\centering
	\begin{subfigure}{0.325\linewidth}
		\centering
		\includegraphics [width=2in,height=2in]{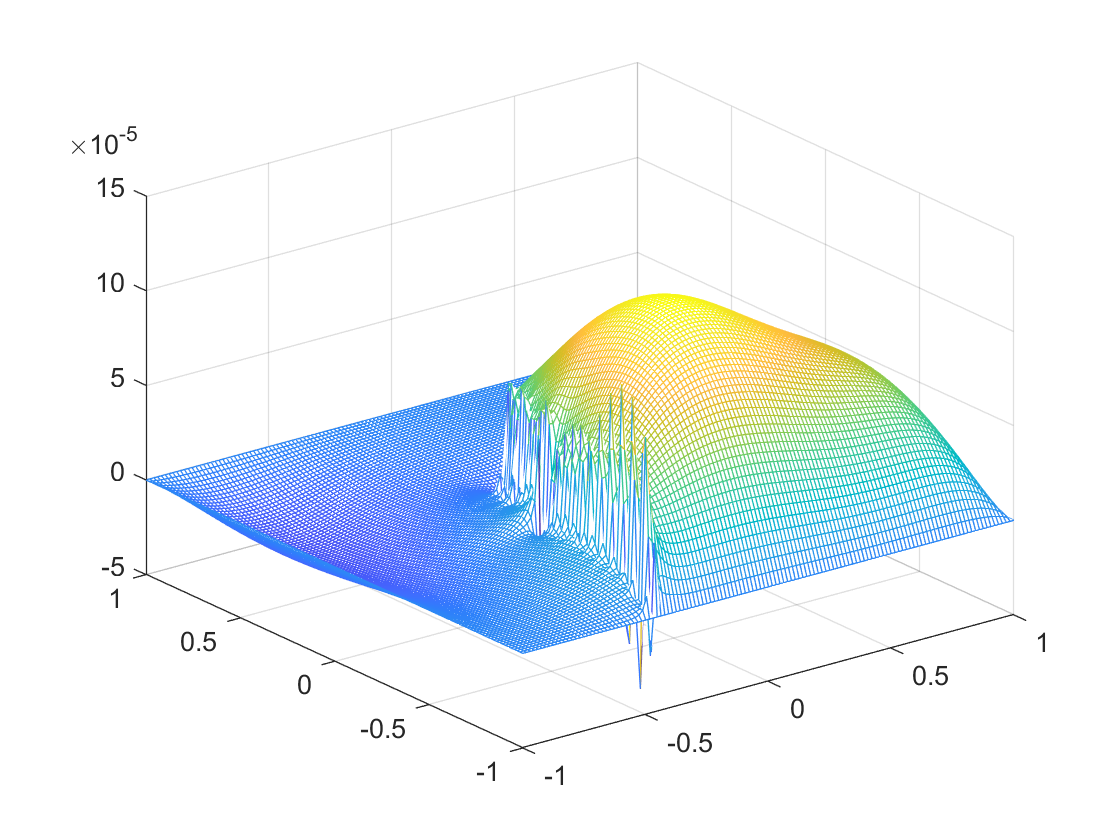}
		\subcaption*{\footnotesize (b) adjoint state error}
		\label{chutian3}
	\end{subfigure}
	\centering
	\begin{subfigure}{0.325\linewidth}
		\centering
		\includegraphics [width=2in,height=2in]{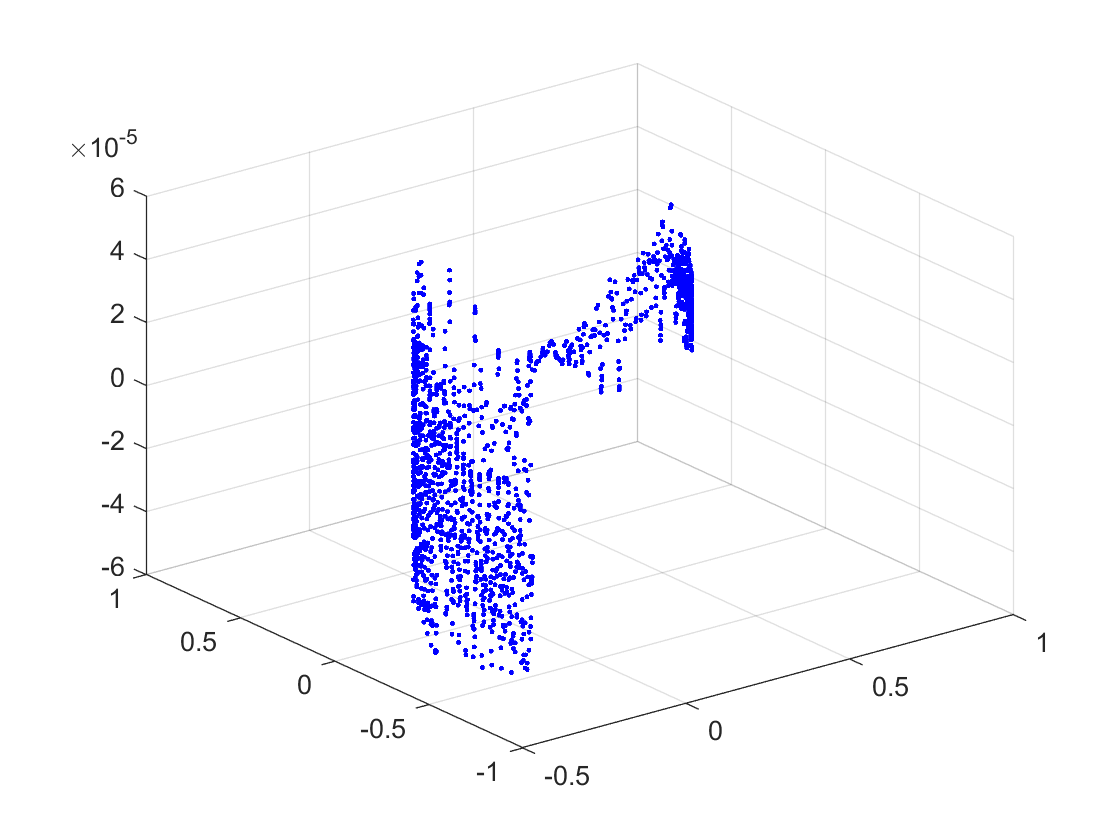}
		\subcaption*{\footnotesize (c) control error}
		\label{chutian3}
	\end{subfigure}
	\caption*{\footnotesize Fig. 12. The error of the state, adjoint state, and control (without control constraints) for Example 2.}
	\label{da_chutian}
\end{figure}
\begin{center}
\includegraphics[width=5in,height=2.5in]{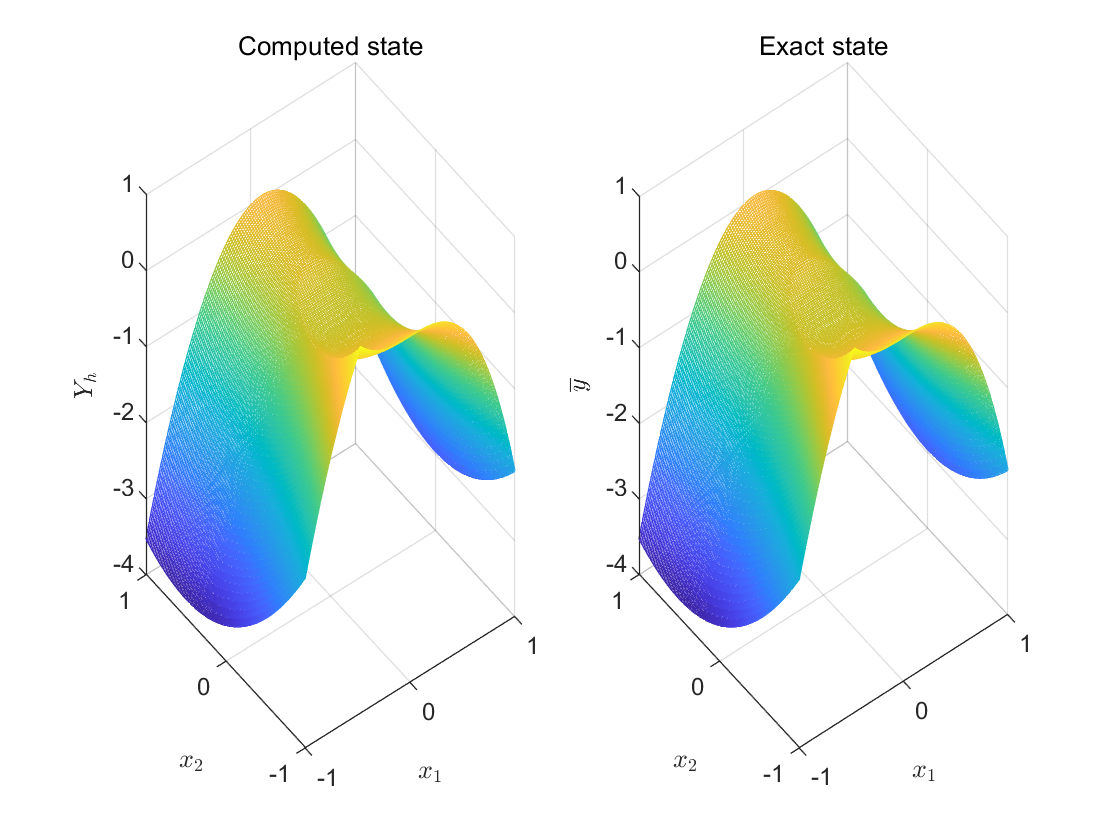}\\
{\mbox{\footnotesize Fig. 13. The computed state and the exact state (with control constraints).}}
\end{center}
\begin{center}
\includegraphics[width=5in,height=2.5in]{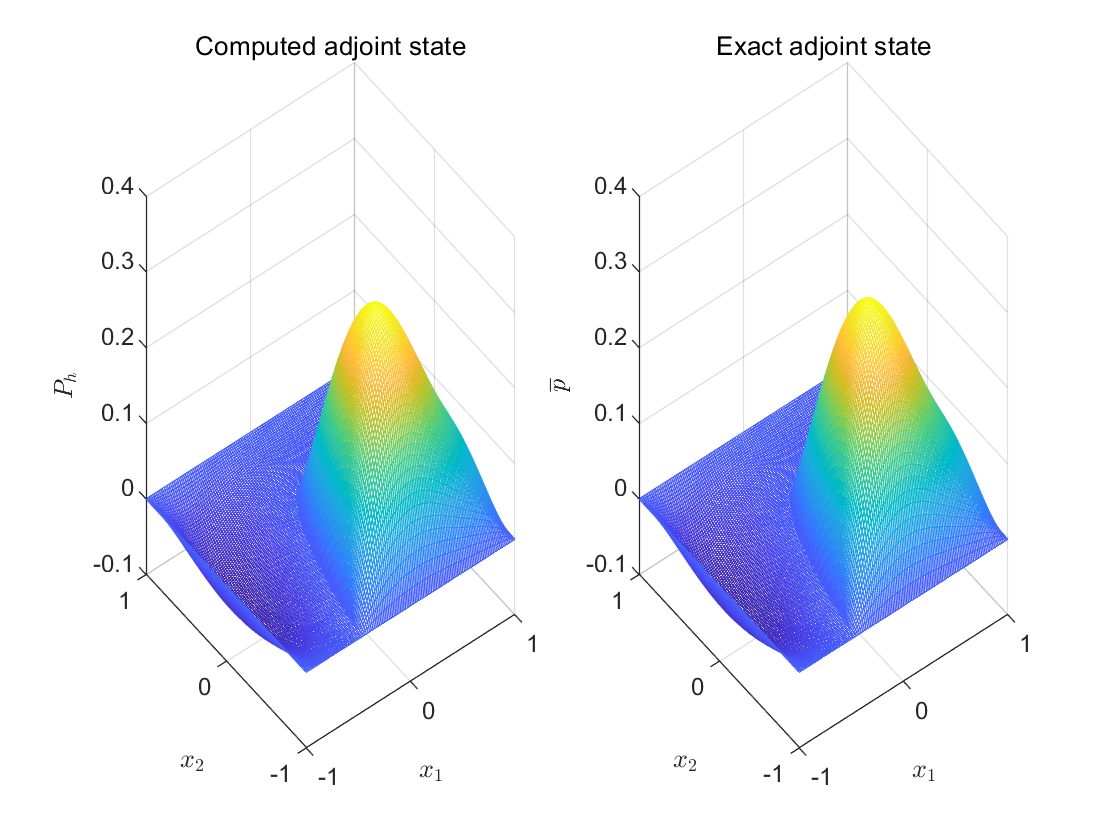}\\
{\mbox{\footnotesize Fig. 14. The computed adjoint state and the exact adjoint state (with control constraints).}}
\end{center}
\begin{center}
\includegraphics[width=5in,height=2.5in]{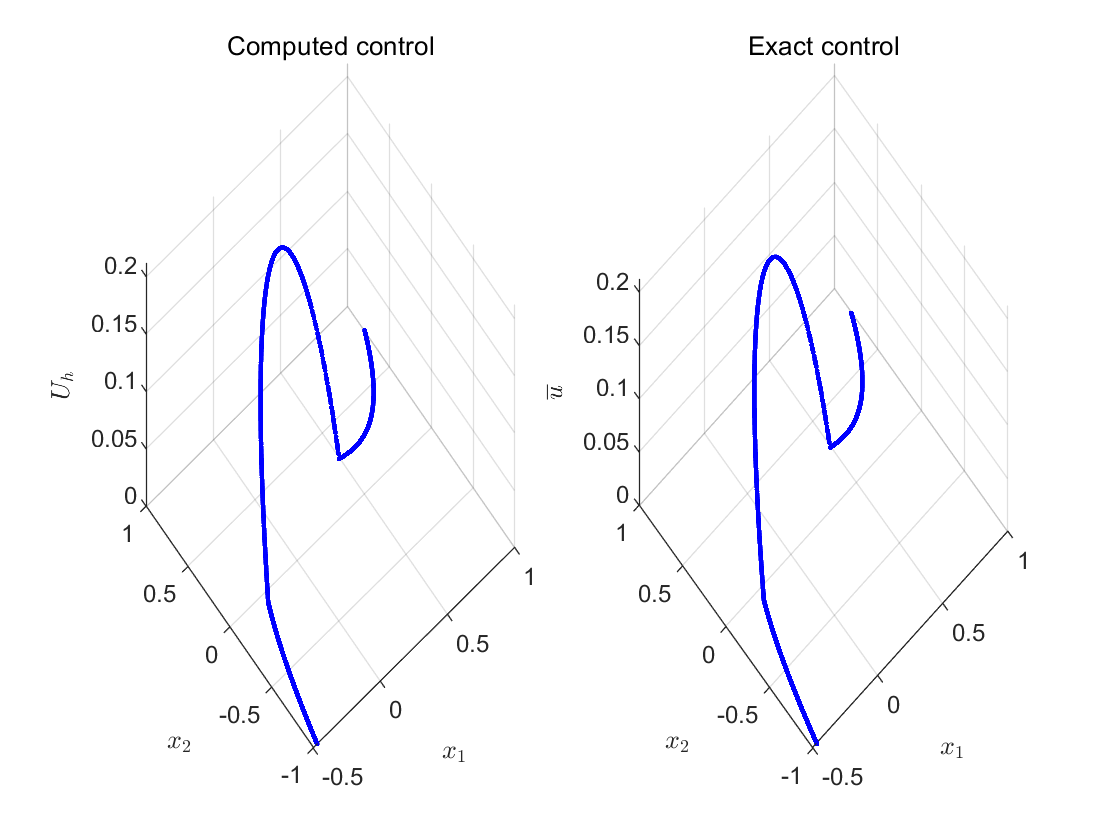}\\
{\mbox{\footnotesize Fig. 15. The computed control and the exact control (with control constraints).}}
\end{center}
\begin{figure}[H]
	\centering
	\begin{subfigure}{0.325\linewidth}
		\centering
		\includegraphics [width=2in,height=2in]{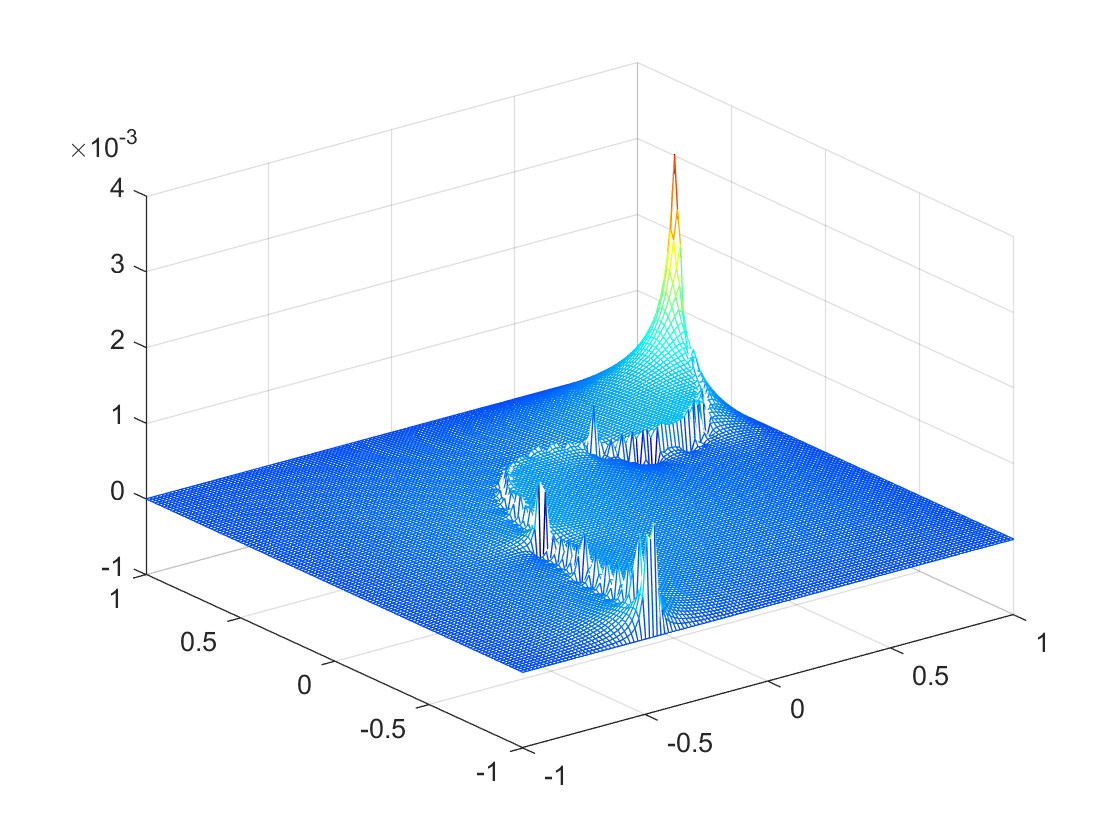}
		\subcaption*{\footnotesize (a) state error}
		\label{chutian3}
	\end{subfigure}
	\centering
	\begin{subfigure}{0.325\linewidth}
		\centering
		\includegraphics [width=2in,height=2in]{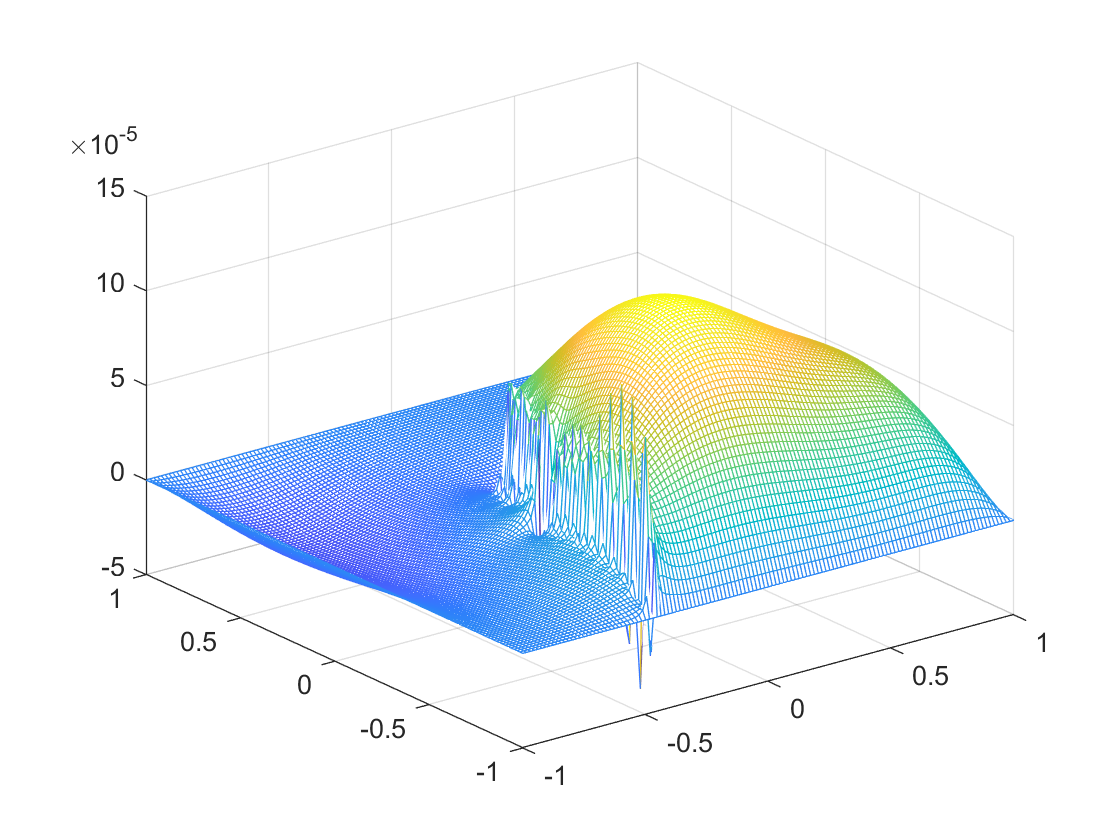}
		\subcaption*{\footnotesize (b) adjoint state error}
		\label{chutian3}
	\end{subfigure}
	\centering
	\begin{subfigure}{0.325\linewidth}
		\centering
		\includegraphics [width=2in,height=2in]{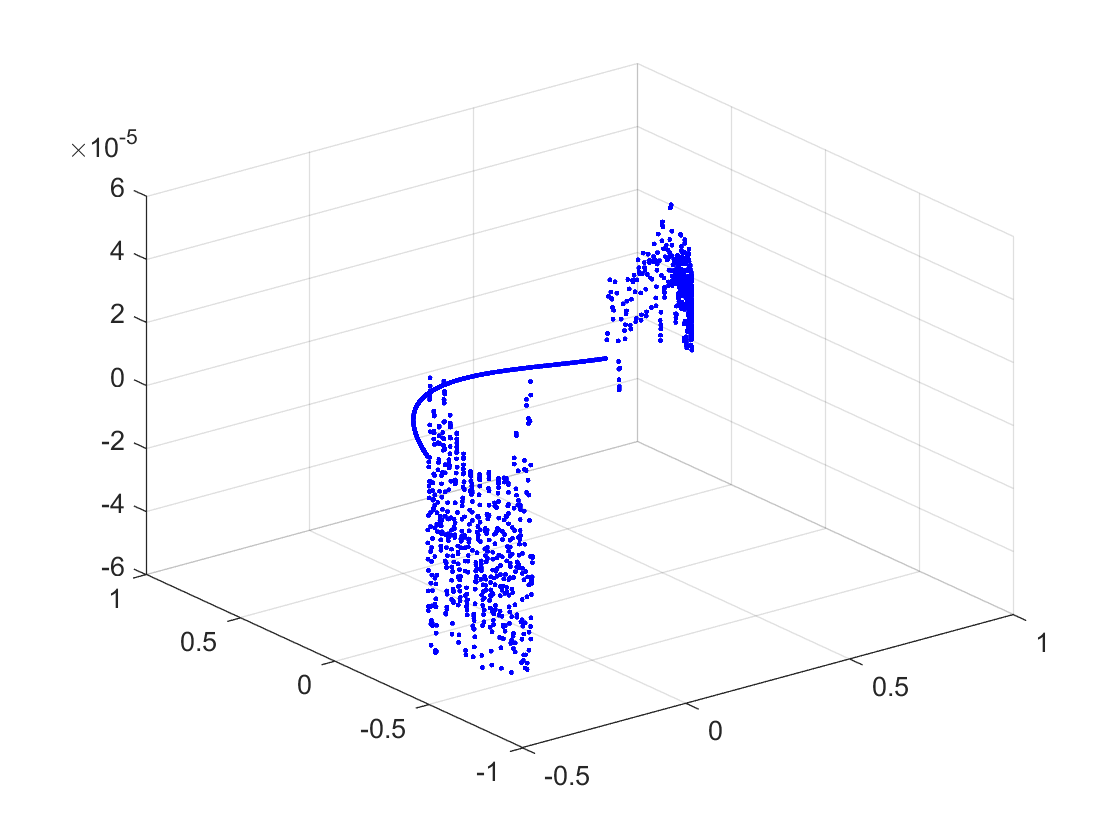}
		\subcaption*{\footnotesize (c) control error}
		\label{chutian3}
	\end{subfigure}
	\caption*{\footnotesize Fig. 16. The error of the state, adjoint state, and control (with control constraints) for Example 2.}
	\label{da_chutian}
\end{figure}
\noindent\textbf{Example 3.} In this example, we consider a more complicated interface: a flower-like shape\cite{Jo2021} (see Fig. 17). The level set function is $ \Gamma=\{(r,\theta): r^4(1+0.4\sin(6\theta))-0.3=0\}$.
\begin{center}
\includegraphics[width=2in,height=2in]{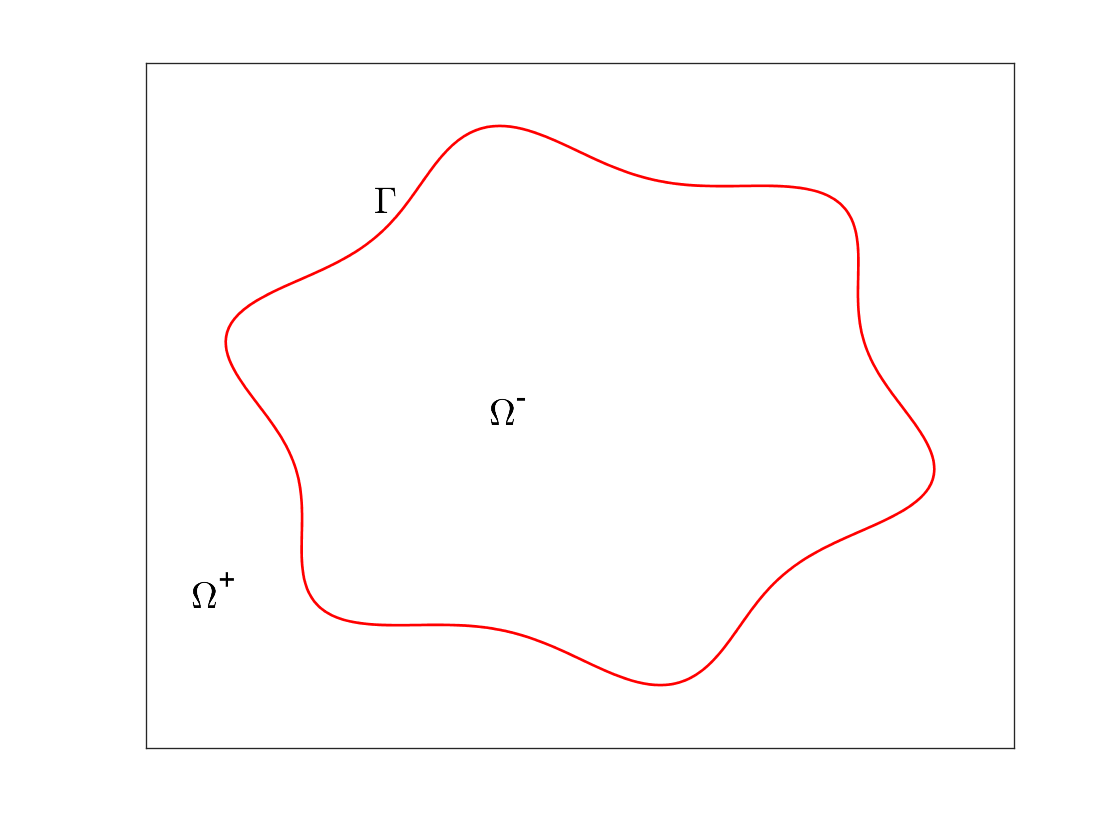}\\
{\mbox{\footnotesize Fig. 17. The flower-like interface for Example 3.}}
\end{center}
The data is chosen as:
\begin{equation*}
\begin{aligned}
&y_{d}=\left\{\begin{array}{ll}
10 & \text { if  }\,\, (x_{1},x_{2})\in \Omega^-, \\
1 & \text { if  }\,\, (x_{1},x_{2}) \in \Omega^+,
\end{array}\right.\\
&f=1 \quad\text{for} ~(x_{1},x_{2}) \in \Omega, \quad g=0 \quad\text{for} ~(x_{1},x_{2}) \in \Gamma, \quad y_{0}=0 \quad\text{for} ~(x_{1},x_{2}) \in \Omega.
\end{aligned}
\end{equation*}
Due to the complex geometry of the interface, it is difficult to give an exact solution. Thus, we use the numerical solutions on the spatial mesh with $N=128$ and temporal mesh with $M=4096$ as a reference solution to show the convergence order. The time steps are taken as $k=O(h^{2})$. The results are shown in Table 9. Fig. 18 shows the images of the numerical solutions for the state, control, and adjoint state with $N=128$ and $M=4096$. From these results we can conclude that our method is also effective for the case of complex interfaces without exact solutions.
\begin{table}[H]
		\centering		
		\caption{ The $L^{2}$ error and convergence order of the state, control, and adjoint state for Example 3 with $\beta^-=1$ and $\beta^+=10$.}
		\label{table11}
		\begin{tabular}{*{7}{c}}
\bottomrule
\multirow{2}*{$1/h$} & \multicolumn{2}{c}{state}&
\multicolumn{2}{c}{control}&
\multicolumn{2}{c}{adjoint state}\\
\cmidrule(lr){2-3}\cmidrule(lr){4-5}\cmidrule(lr){6-7}
    &$\|\overline{y}-Y_{h}\|_{L^{2}(I;L^{2}(\Omega))}$& Order & $\|\overline{u}-U_{h}\|_{L^{2}(I;L^{2}(\Gamma))}$ & Order & $\|\overline{p}-P_{h}\|_{L^{2}(I;L^{2}(\Omega))}$ & Order \\
  \midrule
  4 & $ 1.6829\text{E-02}	$ & $\setminus$&$6.1680\text{E-02}$& $\setminus$ & $ 1.6310 \text{E-02}$& $\setminus$ \\
  8 & $  4.5085\text{E-03}	  $ & $  1.9002 $ & $2.0254\text{E-02}$ & $ 1.6066$& $     	 4.2703 \text{E-02}    $& $   1.9333   $ \\
  16 & $1.0444\text{E-03}	 $ & $  2.1100$ & $ 6.0583\text{E-03}	$ & $ 1.7413$& $   1.0195\text{E-02}	 $& $      2.0665 $ \\
  32 & $ 2.0607\text{E-04}	 $ & $ 2.3414$  & $1.3531\text{E-03}	 $ & $ 2.1626$& $       2.0382 \text{E-03}	   $& $     2.3225    $ \\
  \bottomrule
\end{tabular}
\end{table}
\begin{figure}[H]
	\centering
	\begin{subfigure}{0.325\linewidth}
		\centering
		\includegraphics [width=2in,height=2in]{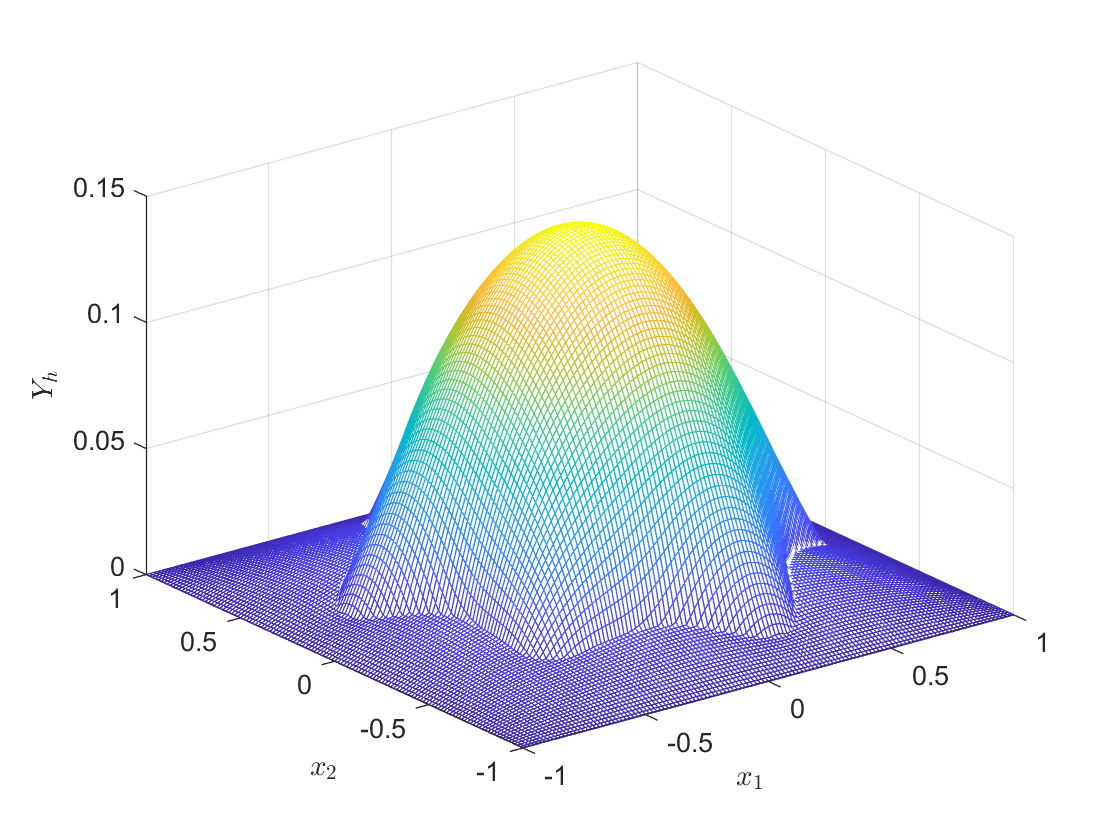}
		\subcaption*{\footnotesize (a) computed state}
		\label{chutian3}
	\end{subfigure}
	\centering
	\begin{subfigure}{0.325\linewidth}
		\centering
		\includegraphics [width=2in,height=2in]{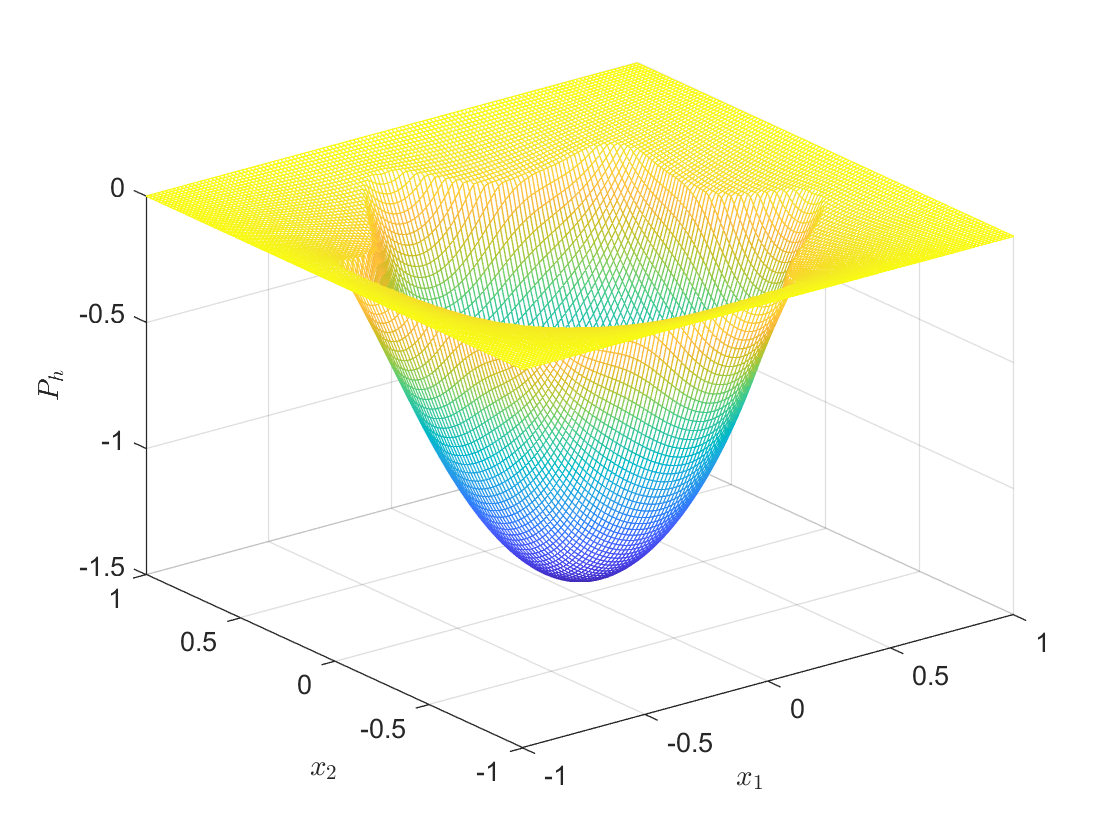}
		\subcaption*{\footnotesize (b) computed adjoint state}
		\label{chutian3}
	\end{subfigure}
	\centering
	\begin{subfigure}{0.325\linewidth}
		\centering
		\includegraphics [width=2in,height=2in]{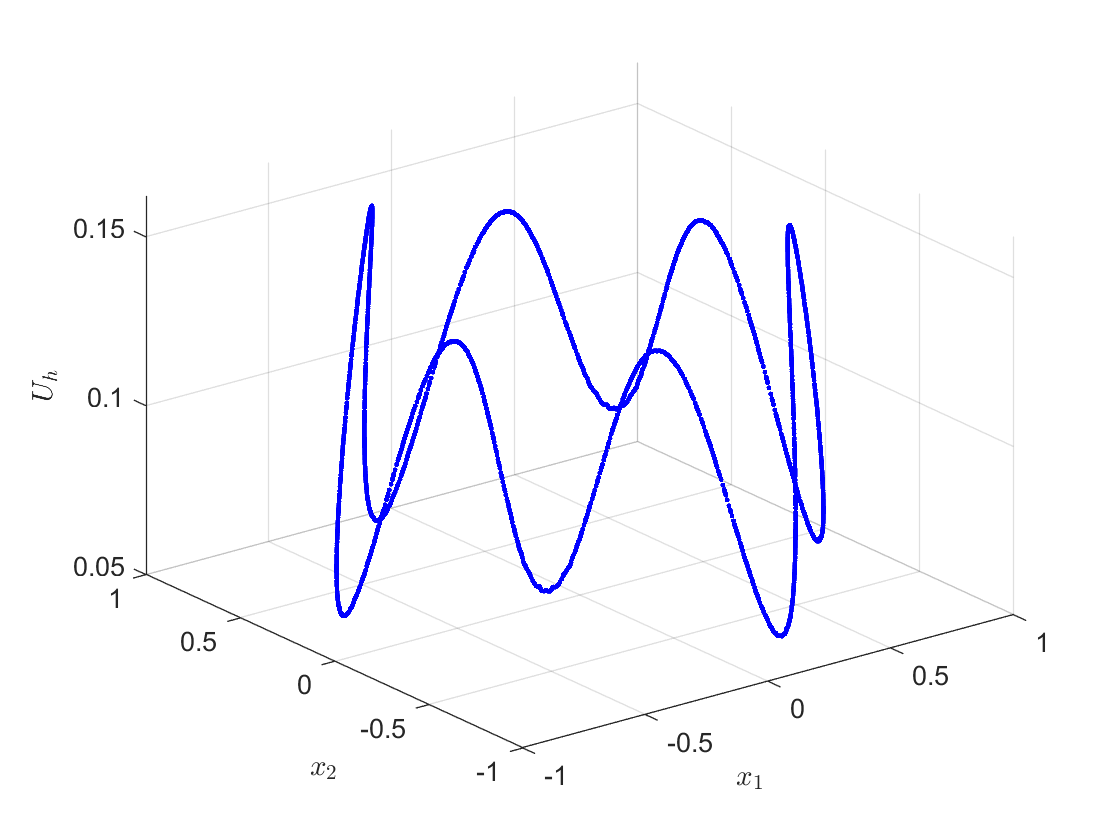}
		\subcaption*{\footnotesize (c) computed control}
		\label{chutian3}
	\end{subfigure}
	\caption*{\footnotesize Fig. 18. The computed state, adjoint state and control with $N=128$ and $M=4096$.}
	\label{da_chutian}
\end{figure}
\section{\bf Conclusion}
In this paper, we have developed an efficient numerical method for optimal control problems governed by parabolic interface problems. Firstly, we derive the optimality conditions for the control problem and the corresponding regularity results. Then, for the control problem, we use the stable generalized finite element method for space discretization and the backward Euler scheme for time discretization of the state and variational discretization for the control variable. Finally, we have obtained a priori error estimates for the fully discretized control problem and provided numerical experiments to support the theoretical results. This approach is a conforming method that does not require any penalty parameters or stability schemes. The method is also easy to implement and can be applied to optimal control problems involving moving interfaces.

\textbf{Acknowledgements}: This work was supported by the National Key R\&D Program of China (2022YFA1004402), the Xinjiang Uygur Autonomous Region Natural Science Fund (2022D01C409), the Natural Science Foundation of Xinjiang Uygur Autonomous Region ( 2025D14015), and the Innovation Project of Excellent Doctoral Students of  Xinjiang University (XJDX2025YJS031).

\end{document}